\documentclass[11pt,reqno,letterpaper]{amsart}

\usepackage{amsmath}
\usepackage{amssymb}
\usepackage{amsthm}
\usepackage{mathtools}
\usepackage{microtype}
\usepackage{booktabs}
\usepackage{enumitem}
\usepackage[letterpaper,margin=1in]{geometry}
\usepackage[hidelinks]{hyperref}

\numberwithin{equation}{section}

\newtheorem{theorem}{Theorem}[section]
\newtheorem{proposition}[theorem]{Proposition}
\newtheorem{lemma}[theorem]{Lemma}

\theoremstyle{definition}

\theoremstyle{remark}
\newtheorem{remark}[theorem]{Remark}

\DeclareMathOperator{\lcm}{lcm}

\title{The Irrationality Measure of $\pi$ Is at Most
  $7.101862832357$}
\author{Yufei Bai}
\date{September 5, 2026}


\begin{document}

\begin{abstract}
We introduce two independent numerator exponents into the
Zeilberger--Zudilin integral and specialize them to
\[
  A_1=A_2=\frac{1857}{2785}.
\]
The resulting integer linear forms in $1$ and $\pi$ prove
\[
  \mu(\pi)<7.101862832357.
\]
This lowers the Zeilberger--Zudilin upper bound
$7.103205334137\ldots$ by more than $0.001342501780$; the difference
between the unrounded bounds is $0.0013425017806509\ldots$,
approximately $0.01890\%$.  The same parameter point is a strict
two-dimensional local minimizer of the explicit auxiliary upper-bound
function in its admissible arithmetic chamber.  This is a local statement
about that function, not a claim that the point is a global optimizer among
all constructions.

\par\smallskip
\noindent\textit{Generative-AI disclosure.}
OpenAI's GPT-5.6 Sol and GPT-6 Astra assisted throughout the research and
preparation of this article, including parameter exploration, proof
development and checking, computation, and manuscript drafting and revision.
Responsibility for all mathematical claims and for the final text rests with
the author.
\end{abstract}

\maketitle

\section{Introduction and main results}\label{sec:introduction}

For an irrational real number $\xi$, its irrationality measure $\mu(\xi)$
is the infimum of the real numbers $\mu$ for which, for every
$\varepsilon>0$, the inequality
\[
  \left|\xi-\frac pq\right|>q^{-\mu-\varepsilon}
\]
holds for all integers $p$ and all sufficiently large positive integers
$q$.  The exact value of $\mu(\pi)$ remains unknown.  Hata's rational
approximations gave an important earlier estimate \cite{hata-1993}, and
Salikhov subsequently obtained the bound $7.606308\ldots$ by a contour
integral with strong arithmetic cancellation \cite{salikhov-2010}.
Zeilberger and Zudilin refined that construction and proved
\[
  \mu(\pi)\leq 7.103205334137\ldots
\]
in \cite{zeilberger-zudilin-2020}.  Their human-readable argument in
Part~II of that paper provides the architecture for the fixed-parameter
proof below; every changed exponent, valuation, prime range, and
asymptotic quantity is recomputed here.

Our starting point is the two-exponent integral
\begin{equation}\label{eq:original-integral}
 I_{A_1,A_2}(N)
 =-i\int_{4-2i}^{4+2i}
 \frac{(x-5)^{A_1N}Q_0(x)^{A_2N}}
 {x^{N+1}(x-10)^{N+1}}\,dx,
\end{equation}
where
\begin{equation}\label{eq:q-zero}
 Q_0(x)=(x-4+2i)(x-4-2i)(x-6+2i)(x-6-2i).
\end{equation}
We normalize rational parameters by
\begin{equation}\label{eq:parameter-normalization}
 \alpha=\frac{A_1}{2}=\frac ac,
 \qquad
 \beta=A_2=\frac bc,
 \qquad
 N=cn.
\end{equation}
The open arithmetic chamber used throughout the paper is
\begin{equation}\label{eq:admissible-chamber}
 \alpha+\beta>1,
 \qquad 2\beta>1,
 \qquad 7\beta<5,
 \qquad 2\alpha+4\beta-2>1.
\end{equation}
For integral $a,b,c$, these strict conditions take the form
\begin{equation}\label{eq:integer-chamber}
 a+b\geq c+1,
 \qquad 2b>c,
 \qquad 7b\leq 5c-1,
 \qquad 2a+4b-2c\geq c+1.
\end{equation}

The parameter triple used for the final estimate is
\begin{equation}\label{eq:candidate-triple}
 (a,b,c)=(1857,3714,5570).
\end{equation}
Consequently
\begin{equation}\label{eq:candidate-parameters}
 (\alpha_*,\beta_*)
 =\left(\frac{1857}{5570},\frac{1857}{2785}\right),
 \qquad
 A_1=A_2=\frac{1857}{2785}.
\end{equation}
This rational point emerged from a search within the diagonal family; the
proof establishes its local behavior, not a closed-form global
optimization principle.

For later use, set $t=x-5$ and define
\begin{equation}\label{eq:rational-function}
 R_n(t)=
 \frac{5t^{2an}(t^4+6t^2+25)^{bn}}
 {(25-t^2)^{cn+1}},
 \qquad
 J_n=i\int_{-1-2i}^{-1+2i}R_n(t)\,dt.
\end{equation}
At the triple \eqref{eq:candidate-triple}, the change of variables gives
$J_n=5I_{A_1,A_2}(5570n)$.  Two arithmetic quantities that recur in the
partial-fraction and denominator estimates are
\begin{equation}\label{eq:d0-h}
 d_0=2a+4b-2c=7430,
 \qquad
 h=5b-\frac{5c}{2}=4645.
\end{equation}
The resulting rational function admits an even partial-fraction
decomposition and hence a linear form
\begin{equation}\label{eq:linear-form-preview}
 J_n=u_n+v_n\pi,
 \qquad u_n,v_n\in\mathbb Q.
\end{equation}
The arithmetic argument clears the denominators of these coefficients,
while the analytic argument separately controls the ordinary exponential
growth of $v_n$ and the exponential decay of $J_n$.

The first principal result is the following strict bound.

\begin{theorem}[Fixed-parameter bound]\label{thm:main-bound}
For the parameter choice \eqref{eq:candidate-parameters},
\[
  \mu(\pi)<7.101862832357.
\]
More precisely, the auxiliary calculation gives
\[
  1+\frac{\sigma}{\tau}
  =7.1018628323563507957\ldots .
\]
\end{theorem}

The passage from exponentially small integer linear forms to an
irrationality-measure estimate uses a Hata-type lemma, which is stated and
proved with the needed nonvanishing and index-selection hypotheses rather
than invoked as a black box; compare \cite{hata-1993}.  The coefficient
estimate uses a positive-coefficient saddle argument of the standard type
developed in \cite[Chapter~VIII]{flajolet-sedgewick-2009}.  Creative
telescoping is not required for the proof, although the
Almkvist--Zeilberger method \cite{almkvist-zeilberger-1990} is part of the
historical route to this family of recurrences.

To formulate the parameter comparison, let $r(\alpha,\beta)$ be the
coefficient-growth rate, let $s(\alpha,\beta)$ be the contour-decay rate,
and let $\mathcal C(\alpha,\beta)$ be the exponential cost remaining after
the removable-prime saving.  On the region where the denominator below is
positive, define the explicit auxiliary bound
\begin{equation}\label{eq:auxiliary-bound}
 \mathcal B(\alpha,\beta)
 =1+\frac{r(\alpha,\beta)+\mathcal C(\alpha,\beta)}
 {-s(\alpha,\beta)-\mathcal C(\alpha,\beta)}.
\end{equation}
All three constituent functions are given explicitly in the body of the
paper.

The two-exponent deformation already improves infinitesimally on the old
diagonal point.  In particular, for all sufficiently small
$\varepsilon>0$,
\begin{equation}\label{eq:old-point-variation}
 \mathcal B\left(
   \left(\frac13,\frac23\right)
   +\varepsilon\left(\frac12,1\right)
 \right)
 <
 \mathcal B\left(\frac13,\frac23\right).
\end{equation}
Thus the Zeilberger--Zudilin exponent is not locally optimal for this
two-parameter auxiliary function.  A rational refinement along this
direction leads to \eqref{eq:candidate-triple}, after which an exact
sector-by-sector variation analysis proves the complementary local
statement.

\begin{theorem}[Strict local minimality]\label{thm:local-minimum}
The point
\[
  (\alpha_*,\beta_*)
  =\left(\frac{1857}{5570},\frac{1857}{2785}\right)
\]
is a strict two-dimensional local minimizer of the function $\mathcal B$
in the admissible chamber \eqref{eq:admissible-chamber}.  Equivalently, in
the original exponent coordinates the locally minimizing point is
$(A_1,A_2)=(1857/2785,1857/2785)$.
\end{theorem}

The proof is organized as follows.  We first derive the partial fractions,
the complete Laurent-coefficient valuations, and the integer linear forms.
We next evaluate the exponential saving from removable primes.  The
positive real saddle for the coefficients and a rigorously controlled
complex contour then give the two analytic rates.  After the
irrationality-measure lemma yields Theorem~\ref{thm:main-bound}, the final
section derives the exponent, proves the logarithmic variation formula,
and checks all twelve directional sectors needed for
Theorem~\ref{thm:local-minimum}.  Exact polynomial and interval calculations
used by the contour and local analyses are printed in the appendix.

We emphasize the scope.  Theorem~\ref{thm:local-minimum} does not prove
global optimality over all positive parameters, optimality among every
possible denominator saving, or optimality among unrelated constructions.
Theorem~\ref{thm:main-bound} is an upper bound and does not determine the
exact value of $\mu(\pi)$.

\section{The rational function and its arithmetic}\label{sec:arithmetic}

In this section $n\geq1$ and $(a,b,c)=(1857,3714,5570)$ are fixed.
Thus $cn$ is even.  We retain the letters $a,b,c$ in the formulas in
order to display where the two numerator exponents enter.  All valuations
of rational numbers are normalized by $\nu_p(p)=1$; on $\mathbb Q(i)$ we
use $\nu_{1+i}(1+i)=1$.

\subsection{The shift and the partial fractions}

Under $t=x-5$, the four noncentral numerator factors in
\eqref{eq:original-integral} have product
\[
 \bigl((t+1)^2+4\bigr)\bigl((t-1)^2+4\bigr)
 =t^4+6t^2+25.
\]
Moreover, $cn+1$ is odd, and hence
\[
 (t+5)^{cn+1}(t-5)^{cn+1}
 =-(5+t)^{cn+1}(5-t)^{cn+1}.
\]
Using $A_1cn=2an$ and $A_2cn=bn$ in
\eqref{eq:original-integral} therefore gives
\begin{equation}\label{eq:shifted-integral-identity}
 I_{A_1,A_2}(5570n)
 =i\int_{-1-2i}^{-1+2i}
 \frac{t^{2an}(t^4+6t^2+25)^{bn}}
 {(25-t^2)^{cn+1}}\,dt,
 \qquad J_n=5I_{A_1,A_2}(5570n).
\end{equation}

The numerator and denominator of $R_n$ are even polynomials.  Division
in $\mathbb Z[t]$ is possible because the leading coefficient of
$(25-t^2)^{cn+1}$ is $-1$.  The quotient is consequently an even
polynomial with integral coefficients.  The two poles are $-5$ and $5$,
both of order $cn+1$, and evenness identifies their corresponding
principal-part coefficients.  Thus there are $c_{j,n}\in\mathbb Q$ such
that
\begin{equation}\label{eq:even-partial-fractions}
 R_n(t)=P_n(t)+\sum_{j=0}^{cn}c_{j,n}
 \left(\frac1{(5+t)^{j+1}}+\frac1{(5-t)^{j+1}}\right),
 \qquad P_n\in\mathbb Z[t^2].
\end{equation}
The degree of the rational function at infinity is
\[
 (2a+4b)n-(2cn+2)=d_0n-2,
\]
so $\deg P_n=d_0n-2=7430n-2$.

The $j=0$ term is the only logarithmic term.  Along the indicated
vertical segment, with the logarithm continued from its initial value,
\begin{equation}\label{eq:logarithmic-integral}
 \int_{-1-2i}^{-1+2i}
 \left(\frac1{5+t}+\frac1{5-t}\right)dt
 =\log\frac{(4+2i)(6+2i)}{(4-2i)(6-2i)}
 =\frac{\pi i}{2}.
\end{equation}
Every other term in the integrated decomposition is rational.  It follows
already that
\begin{equation}\label{eq:rational-linear-form}
 J_n=u_n+v_n\pi,
 \qquad u_n,v_n\in\mathbb Q,
 \qquad v_n=-\frac{c_{0,n}}2.
\end{equation}
We now prove the denominator statement needed to turn this into an integer
linear form.  The order of the next six lemmas parallels the arithmetic
argument of Zeilberger--Zudilin, but the removable-prime lemma below uses
an exact derivative in a finite field.

\subsection{Laurent coefficients at the pole}

For $m\geq0$ put
\[
 \mathcal D_m f=\left.\frac1{m!}\frac{d^m}{dt^m}f(t)\right|_{t=-5}.
\]
For every integer $j\leq cn$, define $c_{j,n}$ to be the coefficient of
$(t+5)^{-j-1}$ in the Laurent expansion of $R_n$ at $-5$.  For
$0\leq j\leq cn$ this agrees with \eqref{eq:even-partial-fractions}; the
negative Laurent indices will be essential for the polynomial part.

\begin{lemma}[Complete Laurent valuation]\label{lem:laurent-valuations}
For $0\leq j\leq cn$,
\begin{equation}\label{eq:positive-laurent-valuations}
 \nu_2(c_{j,n})\geq hn+\left\lfloor\frac{3j}{2}\right\rfloor-1,
 \qquad
 \nu_5(c_{j,n})\geq2(a+b-c)n+j.
\end{equation}
For the complete range
\begin{equation}\label{eq:needed-laurent-range}
 -d_0n+1\leq j\leq cn
\end{equation}
one has
\begin{equation}\label{eq:scaled-laurent-valuations}
 \nu_2(10^{-j}c_{j,n})\geq(4b-2c)n-1,
 \qquad
 \nu_5(10^{-j}c_{j,n})\geq2(a+b-c)n.
\end{equation}
In particular, at the chosen triple,
\begin{equation}\label{eq:scaled-laurent-integrality}
 \nu_2(10^{-j}c_{j,n})\geq3716n-1,
 \quad \nu_5(10^{-j}c_{j,n})\geq2n,
 \quad 10^{-j}c_{j,n}\in\mathbb Z.
\end{equation}
\end{lemma}

\begin{proof}
The coefficient is obtained by applying $\mathcal D_{cn-j}$ to
$(t+5)^{cn+1}R_n(t)$.  Let $\mathcal M_j$ be the set of
$\mathbf m=(m_0,\ldots,m_5)\in\mathbb Z_{\geq0}^6$ satisfying
\begin{equation}\label{eq:multi-index-set}
 m_0+\cdots+m_5=cn-j,
 \quad m_1\leq2an,
 \quad m_2,m_3,m_4,m_5\leq bn,
\end{equation}
and set
\[
 T(\mathbf m)=\binom{cn+m_0}{m_0}\binom{2an}{m_1}
 \prod_{\ell=2}^5\binom{bn}{m_\ell}\in\mathbb Z.
\]
The normalized Leibniz formula gives the exact finite expansion
\begin{align}
 c_{j,n}=5\sum_{\mathbf m\in\mathcal M_j}&T(\mathbf m)
 10^{-cn-1-m_0}(-5)^{2an-m_1}
 (-4+2i)^{bn-m_2}(-4-2i)^{bn-m_3}\notag\\
 &\hspace{31mm}{}
 \times(-6+2i)^{bn-m_4}(-6-2i)^{bn-m_5}.
 \label{eq:leibniz-laurent}
\end{align}
This formula remains valid for every negative $j$: only $m_0$ is then
allowed to exceed the numerator exponents.

To read off its denominators, use
\[
 4\mp2i=2(2\mp i),
 \qquad 6-2i=2(1-i)(2+i),
 \qquad 6+2i=2(1+i)(2-i).
\]
After pairing conjugate factors, the summand indexed by $\mathbf m$ in
\eqref{eq:leibniz-laurent} is exactly
\begin{align}
 (-1)^{m_1+\cdots+m_5}T(\mathbf m)&
 2^{(5b-2c)n-1+j+m_1}
 5^{2(a+b-c)n+j}
 (1-i)^{-m_4}(1+i)^{-m_5}\notag\\
 &{}\times(2+i)^{m_2+m_5}(2-i)^{m_3+m_4}.
 \label{eq:factored-laurent-summand}
\end{align}
Indeed, the relation in \eqref{eq:multi-index-set} accounts for both
displayed rational powers.

At the Gaussian prime $1+i$, each summand in
\eqref{eq:factored-laurent-summand} has valuation at least
\begin{equation}\label{eq:gaussian-two-bound}
 (10b-4c)n+2j+2m_1-m_4-m_5-2.
\end{equation}
If $j\geq0$, then $m_4+m_5\leq cn-j$, so this is at least
\begin{equation}\label{eq:strong-gaussian-two-bound}
 (10b-5c)n+3j-2=2hn+3j-2.
\end{equation}
Because $c_{j,n}\in\mathbb Q$, one has
$\nu_{1+i}(c_{j,n})=2\nu_2(c_{j,n})$.  Taking account of the integrality
of $\nu_2$ in \eqref{eq:strong-gaussian-two-bound} gives the first
inequality in \eqref{eq:positive-laurent-valuations} (with harmless slack
when $j$ is odd).

At each of the conjugate primes $2+i$ and $2-i$ above $5$, the valuation
of every summand is at least $2(a+b-c)n+j$.  Since a rational number has
the same valuation at both primes, the second inequality in
\eqref{eq:positive-laurent-valuations} follows.

Finally multiply \eqref{eq:factored-laurent-summand} by $10^{-j}$.  At
$1+i$ the $2j$ in \eqref{eq:gaussian-two-bound} disappears, and
$m_4+m_5\leq2bn$ gives
\[
 \nu_{1+i}(10^{-j}c_{j,n})
 \geq2(4b-2c)n-2.
\]
At either prime above $5$, multiplication by $10^{-j}$ changes the lower
bound to $2(a+b-c)n$.  This proves
\eqref{eq:scaled-laurent-valuations}.  Formula
\eqref{eq:factored-laurent-summand} shows that no rational primes other
than $2$ and $5$ can occur in a denominator.  The two lower bounds are
nonnegative for our parameters and $n\geq1$, and hence
$10^{-j}c_{j,n}\in\mathbb Z$ throughout \eqref{eq:needed-laurent-range}.
\end{proof}

\subsection{Deletion of the extra primes}

Define
\begin{equation}\label{eq:deleted-prime-set}
 \begin{split}
 \mathcal P_n=\biggl\{p\text{ prime}:{}&
 \max\{5,\sqrt{d_0n}\}<p\leq d_0n,\\
 &\left\{\frac{an}{p}+\frac12\right\}
 +2\left\{\frac{bn}{p}\right\}
 <\left\{\frac{cn}{p}\right\}\biggr\},
 \qquad \Phi_n=\prod_{p\in\mathcal P_n}p ,
\end{split}
\end{equation}
with the empty product understood to be $1$.

\begin{lemma}[Finite-field deletion]\label{lem:finite-field-deletion}
Let $p\in\mathcal P_n$.  If $p\mid j$, then
\begin{equation}\label{eq:deleted-prime-divisibility}
 c_{j,n}\equiv0\pmod p
\end{equation}
for every $j$ in \eqref{eq:needed-laurent-range}.  Equivalently,
$p\mid10^{-j}c_{j,n}$ in $\mathbb Z$.  In particular,
$\Phi_n\mid c_{0,n}$, and the divisibility for any nonzero $j$ holds
simultaneously for all deleted primes that divide $j$.
\end{lemma}

\begin{proof}
All the Laurent coefficients are $p$-integral by
\eqref{eq:factored-laurent-summand}, because $p>5$.  It is therefore
legitimate to reduce their local expansion modulo $p$.

Choose representatives $A,B,C\in\{0,\ldots,p-1\}$ and nonnegative
integers $q_a,q_b,q_c$ so that
\[
 an=q_ap+A,\qquad bn=q_bp+B,\qquad cn=q_cp+C.
\]
In $\mathbb F_p(t)$ put
\begin{equation}\label{eq:finite-field-U-H}
 U(t)=\frac{5t^{2q_a}(t^4+6t^2+25)^{q_b}}
 {(25-t^2)^{q_c+1}},
 \qquad
 H(t)=t^{2A}(t^4+6t^2+25)^B(25-t^2)^{p-1-C}.
\end{equation}
Frobenius and $5^p=5$ give the exact factorization
\begin{equation}\label{eq:frobenius-factorization}
 R_n(t)=U(t)^pH(t).
\end{equation}

Write $y=t^2$.  As a polynomial in $y$, the support of $H$ is contained
in the integer interval
\begin{equation}\label{eq:H-support}
 [A,\ A+2B+p-1-C].
\end{equation}
The obstruction to integrating a monomial $t^r$ in characteristic $p$
is $r\equiv-1\pmod p$.  For an even polynomial these obstructed
monomials are precisely
\begin{equation}\label{eq:even-integration-obstructions}
 y^{(p-1)/2+\ell p}\qquad(\ell\geq0).
\end{equation}

It remains to check that \eqref{eq:H-support} misses all the exponents in
\eqref{eq:even-integration-obstructions}.  Since $p$ is odd, exactly one
of the following cases occurs.  If $A\leq(p-1)/2$, then
\[
 \left\{\frac{an}{p}+\frac12\right\}=\frac Ap+\frac12.
\]
The strict inequality defining $\mathcal P_n$ implies
\[
 C\geq A+2B+\frac{p+1}{2},
 \qquad
 A+2B+p-1-C\leq\frac{p-3}{2}.
\]
Thus the support ends before the first obstruction.  If instead
$A\geq(p+1)/2$, then
\[
 \left\{\frac{an}{p}+\frac12\right\}=\frac Ap-\frac12,
\]
and the same strict inequality gives
\[
 C\geq A+2B-\frac{p-1}{2},
 \qquad
 A+2B+p-1-C\leq\frac{3p-3}{2}.
\]
Now the support starts strictly after $(p-1)/2$ and ends strictly before
$(3p-1)/2$.  This also avoids every obstruction.

Consequently the coefficient of $t^r$ in $H$ is zero whenever
$p\mid r+1$.  Termwise division by $r+1$ for all remaining monomials
constructs a polynomial $V\in\mathbb F_p[t]$ with $V'=H$.  Since the
derivative of a $p$th power is zero in characteristic $p$,
\begin{equation}\label{eq:finite-field-exact-derivative}
 R_n=U^pV'=(U^pV)'
 \qquad\text{in }\mathbb F_p(t).
\end{equation}
Let $z=t+5$.  In the local Laurent expansion of a rational function
$W=\sum_r w_rz^r$, the coefficient of $z^{-j-1}$ in $W'$ is
$-j w_{-j}$.  It vanishes in $\mathbb F_p$ whenever $p\mid j$.
Applying this to $W=U^pV$ and using
\eqref{eq:finite-field-exact-derivative} proves
\eqref{eq:deleted-prime-divisibility}.  For $j=0$ every prime in
$\mathcal P_n$ applies; distinctness of the primes gives
$\Phi_n\mid c_{0,n}$.
\end{proof}

\begin{lemma}[The reduced least common multiple]\label{lem:reduced-lcm}
Let
\begin{equation}\label{eq:two-lcms}
 L_n^{(0)}=\lcm(1,2,\ldots,d_0n)=\lcm(1,2,\ldots,7430n),
 \qquad L_n=\frac{L_n^{(0)}}{\Phi_n}.
\end{equation}
Then $L_n\in\mathbb Z$ and
\begin{equation}\label{eq:reduced-lcm-divisibility}
 L_n\frac{10^{-j}c_{j,n}}{j}\in\mathbb Z
 \quad(-d_0n+1\leq j\leq cn,\ j\ne0),
 \qquad \frac{c_{0,n}}{\Phi_n}\in\mathbb Z.
\end{equation}
\end{lemma}

\begin{proof}
For $p\in\mathcal P_n$, the inequalities
$\sqrt{d_0n}<p\leq d_0n$ show that $p$ occurs exactly once in
$L_n^{(0)}$.  Hence $\Phi_n\mid L_n^{(0)}$.

Fix a nonzero $j$ in the stated range and put
$q_j=10^{-j}c_{j,n}\in\mathbb Z$.  If a prime divisor of $\Phi_n$ does
not divide $j$, its deletion from $L_n^{(0)}$ is irrelevant to the
denominator $j$.  If $p\in\mathcal P_n$ divides $j$, then $p^2>d_0n
\geq|j|$, so $\nu_p(j)=1$, while Lemma~\ref{lem:finite-field-deletion}
gives $p\mid q_j$.  At all primes not deleted, the elementary inclusion
$L_n^{(0)}/j\in\mathbb Z$ supplies the required valuation.  This proves
the first assertion in \eqref{eq:reduced-lcm-divisibility}; the second is
the $j=0$ assertion of Lemma~\ref{lem:finite-field-deletion}, together
with Lemma~\ref{lem:laurent-valuations}.
\end{proof}

\subsection{The polynomial part}

Set $t_0=-1-2i$.  There are Gaussian integers $B_k$ such that
\begin{equation}\label{eq:endpoint-polynomial-expansion}
 P_n(t)=\sum_{k=0}^{d_0n-2}B_k(t-t_0)^k,
 \qquad B_k\in\mathbb Z[i],
\end{equation}
because $P_n\in\mathbb Z[t]$ and $t_0\in\mathbb Z[i]$.

\begin{lemma}[Endpoint coefficients]\label{lem:endpoint-coefficients}
For $0\leq k\leq d_0n-2$,
\begin{equation}\label{eq:endpoint-coefficient-bound}
 2^{-hn+\lceil3k/2\rceil+3}B_k\in\mathbb Z[i].
\end{equation}
\end{lemma}

\begin{proof}
If $k<bn$, then $R_n$ has a zero of order $bn$ at $t_0$.  Differentiating
\eqref{eq:even-partial-fractions} $k$ times in normalized form therefore
gives
\begin{equation}\label{eq:endpoint-Bk-formula}
 B_k=-\sum_{j=0}^{cn}\binom{j+k}{k}c_{j,n}
 \left(\frac{(-1)^k}{(4-2i)^{j+k+1}}
       +\frac1{(6+2i)^{j+k+1}}\right).
\end{equation}
The proof of Lemma~\ref{lem:laurent-valuations} gave the stronger
Gaussian estimate
\[
 \nu_{1+i}(c_{j,n})\geq2hn+3j-2.
\]
Since $\nu_{1+i}(4-2i)=2$ and
$\nu_{1+i}(6+2i)=3$, the two summands inside
\eqref{eq:endpoint-Bk-formula} have valuation at least, respectively,
\[
 2hn+j-2k-4
 \quad\text{and}\quad
 2hn-3k-5.
\]
Both are at least $2hn-3k-5$, so
$\nu_{1+i}(B_k)\geq2hn-3k-5$.  This implies
\eqref{eq:endpoint-coefficient-bound}.

If $k\geq bn$, no divisibility beyond $B_k\in\mathbb Z[i]$ is needed:
the numerical inequality
\[
 h=4645<5571=\frac{3b}{2}
\]
shows that $-hn+\lceil3k/2\rceil+3\geq0$.  Thus
\eqref{eq:endpoint-coefficient-bound} holds in this range as well.
\end{proof}

\begin{lemma}[Integral of the polynomial part]\label{lem:polynomial-integral}
One has
\begin{equation}\label{eq:polynomial-integrality}
 2^{1-hn}L_n i\int_{-1-2i}^{-1+2i}P_n(t)\,dt\in\mathbb Z.
\end{equation}
\end{lemma}

\begin{proof}
We use two representations of the polynomial part.  First,
\eqref{eq:endpoint-polynomial-expansion} gives
\begin{equation}\label{eq:first-polynomial-integral}
 i\int_{-1-2i}^{-1+2i}P_n(t)\,dt
 =i\sum_{k=0}^{d_0n-2}\frac{B_k(4i)^{k+1}}{k+1}.
\end{equation}
Here $L_n^{(0)}/(k+1)\in\mathbb Z$.  After multiplication by
$2^{1-hn}L_n^{(0)}$, the power of $2$ remaining beyond
\eqref{eq:endpoint-coefficient-bound} is
\[
 2k-\left\lceil\frac{3k}{2}\right\rceil\geq0.
\]
Thus the result belongs to $\mathbb Z[i]$.  The integral in
\eqref{eq:first-polynomial-integral} is rational: an antiderivative of
$P_n\in\mathbb Z[t]$ takes conjugate values at the two endpoints, and
multiplication by $i$ makes their difference real.  Therefore
\begin{equation}\label{eq:polynomial-first-inclusion}
 2^{1-hn}L_n^{(0)}i\int_{-1-2i}^{-1+2i}P_n(t)\,dt\in\mathbb Z.
\end{equation}

For the second representation, put $z=t+5$.  The complete Laurent
expansion of $R_n$ at $-5$ is
\[
 R_n(t)=\sum_{j=-\infty}^{cn}c_{j,n}z^{-j-1},
\]
whereas
\[
 \frac{c_{j,n}}{(5-t)^{j+1}}
 =\frac{c_{j,n}}{(10-z)^{j+1}}
 =\sum_{k\geq1}\binom{j+k-1}{j}
   \frac{c_{j,n}}{10^{j+k}}z^{k-1}.
\]
Comparison with \eqref{eq:even-partial-fractions}, using
$\deg P_n=d_0n-2$, yields the finite identity
\begin{equation}\label{eq:second-polynomial-representation}
 P_n(t)=\sum_{k=1}^{d_0n-1}\left(
 c_{-k,n}-\sum_{j=0}^{cn}\binom{j+k-1}{j}
 \frac{c_{j,n}}{10^{j+k}}\right)z^{k-1}.
\end{equation}

We claim that after termwise integration, multiplication by $L_n$ puts
every summand in $\mathbb Z[i,1/10]$.  For the first term in parentheses,
Lemma~\ref{lem:reduced-lcm} at the negative Laurent index $-k$ gives
\[
 L_n\frac{c_{-k,n}}k\in10^{-k}\mathbb Z.
\]
For $j\geq1$, use
\begin{equation}\label{eq:binomial-denominator-transfer}
 \frac1k\binom{j+k-1}{j}
 =\frac1j\binom{j+k-1}{j-1}
\end{equation}
and Lemma~\ref{lem:reduced-lcm}; the corresponding coefficient again
lies in $10^{-k}\mathbb Z$.  For $j=0$,
\[
 L_n\frac{c_{0,n}}{k10^k}
 =\frac{L_n^{(0)}}k\frac{c_{0,n}}{\Phi_n}\,10^{-k}
 \in10^{-k}\mathbb Z.
\]
The endpoint difference $(4+2i)^k-(4-2i)^k$ is a Gaussian integer.
Consequently the second representation proves
\begin{equation}\label{eq:polynomial-second-inclusion}
 L_n\,i\int_{-1-2i}^{-1+2i}P_n(t)\,dt
 \in\mathbb Z[1/10].
\end{equation}

Let $X$ denote the left side of \eqref{eq:polynomial-integrality}.
Equation \eqref{eq:polynomial-first-inclusion} says $\Phi_nX\in\mathbb Z$,
while \eqref{eq:polynomial-second-inclusion} says
$X\in\mathbb Z[1/10]$.  Since $(\Phi_n,10)=1$,
\begin{equation}\label{eq:intersection-argument}
 \Phi_n^{-1}\mathbb Z\cap\mathbb Z[1/10]=\mathbb Z.
\end{equation}
Hence $X\in\mathbb Z$, proving the lemma.  This intersection is why the
negative Laurent coefficients and both representations of the polynomial
part cannot be omitted.
\end{proof}

\subsection{The pole terms and the integer linear form}

\begin{lemma}[Nonlogarithmic pole terms]\label{lem:nonlogarithmic-poles}
The part of \eqref{eq:even-partial-fractions} with $1\leq j\leq cn$
satisfies
\begin{equation}\label{eq:nonlogarithmic-integrality}
 2^{1-hn}L_n i\int_{-1-2i}^{-1+2i}
 \sum_{j=1}^{cn}c_{j,n}
 \left(\frac1{(5+t)^{j+1}}+\frac1{(5-t)^{j+1}}\right)dt
 \in\mathbb Z.
\end{equation}
\end{lemma}

\begin{proof}
Set
\[
 q_j=L_n\frac{10^{-j}c_{j,n}}j\in\mathbb Z
\]
by Lemma~\ref{lem:reduced-lcm}.  Direct integration gives the expression
inside \eqref{eq:nonlogarithmic-integrality}, before its leading power of
$2$, as
\begin{equation}\label{eq:integrated-pole-expression}
 i\sum_{j=1}^{cn}q_j\left(
 (2+i)^j-(2-i)^j
 -\left(\frac{2+i}{1+i}\right)^j
 +\left(\frac{2-i}{1-i}\right)^j\right).
\end{equation}
Indeed,
\[
 \frac{10}{4-2i}=2+i,
 \quad \frac{10}{4+2i}=2-i,
 \quad \frac{10}{6+2i}=\frac{2+i}{1+i},
 \quad \frac{10}{6-2i}=\frac{2-i}{1-i}.
\]
The proof of Lemma~\ref{lem:laurent-valuations} also gives
\[
 \nu_{1+i}(10^{-j}c_{j,n})\geq2hn+j-2.
\]
Since $\Phi_n$ is odd and $L_n^{(0)}/j$ is $2$-adically integral,
$\nu_{1+i}(L_n/j)\geq0$.  It follows that
\[
 \nu_{1+i}\left(
 2^{1-hn}\frac{q_j}{(1\pm i)^j}\right)\geq
 (2-2hn)+(2hn+j-2)-j=0.
\]
Thus multiplication of \eqref{eq:integrated-pole-expression} by
$2^{1-hn}$ gives a Gaussian integer.  The integrated expression is
rational, by the same conjugation argument as for the polynomial part,
and is therefore an ordinary integer.
\end{proof}

\begin{proposition}[Arithmetic normalization]\label{prop:arithmetic-normalization}
Define
\begin{equation}\label{eq:arithmetic-multiplier}
 M_n=2^{4-4645n}\frac{L_n^{(0)}}{\Phi_n}
 =2^{4-hn}L_n.
\end{equation}
Then there are $U_n,V_n\in\mathbb Z$ such that
\begin{equation}\label{eq:integer-linear-forms}
 \boxed{M_nJ_n=U_n+V_n\pi.}
\end{equation}
More explicitly,
\begin{equation}\label{eq:integer-pi-coefficient}
 V_n=-2^{3-hn}L_n c_{0,n}\in\mathbb Z.
\end{equation}
\end{proposition}

\begin{proof}
Lemmas~\ref{lem:polynomial-integral} and
\ref{lem:nonlogarithmic-poles} show that multiplication by
$2^{1-hn}L_n$ makes the polynomial integral and all nonlogarithmic pole
integrals integers.  The multiplier $M_n$ is eight times this common
multiplier, so those contributions remain integral.

It remains to treat the logarithmic coefficient.  Equations
\eqref{eq:logarithmic-integral} and \eqref{eq:rational-linear-form} show
that its coefficient of $\pi$ is $-c_{0,n}/2$.  By
Lemma~\ref{lem:reduced-lcm}, $c_{0,n}/\Phi_n\in\mathbb Z$; since
$\Phi_n$ is odd, Lemma~\ref{lem:laurent-valuations} also gives
\[
 \nu_2(c_{0,n}/\Phi_n)=\nu_2(c_{0,n})\geq hn-1.
\]
Consequently
\[
 -\frac{M_nc_{0,n}}2
 =-2^{3-hn}L_n^{(0)}\frac{c_{0,n}}{\Phi_n}
 \in\mathbb Z,
\]
which is \eqref{eq:integer-pi-coefficient}.  Combining the three
separately cleared contributions proves \eqref{eq:integer-linear-forms}.
\end{proof}

\section{The removable-prime saving}\label{sec:prime-saving}

We now determine the exponential size of the product \(\Phi_n\) from
\eqref{eq:deleted-prime-set}.  The argument has three parts: an exact
description of the periodic set that selects the primes, a prime-number-
theorem passage, and a numerical enclosure whose error terms are printed
explicitly.

\subsection{Exact decomposition of the periodic set}

Put
\begin{equation}\label{eq:periodic-set-E}
 E=\left\{u\in[0,1):
 \left\{1857u+\frac12\right\}+2\{3714u\}<\{5570u\}
 \right\}.
\end{equation}
For this subsection write
\[
 a=1857,\qquad b=3714=2a,\qquad c=5570=\frac{3b}{2}-1,
 \qquad d=a+2b-c=3715=b+1.
\]

\begin{proposition}[The five interval families]\label{prop:E-decomposition}
The set \(E\) is the following disjoint union of half-open intervals:
\begin{align}
 &\bigcup_{\substack{1\leq j\leq371\\j\text{ odd}}}
 \left[\frac{j}{3714},\frac{3j+1}{2\cdot5570}\right),
 &&\bigcup_{\substack{2\leq j\leq1114\\j\text{ even}}}
 \left[\frac{j}{3714},\frac{3j}{2\cdot5570}\right),
 \label{eq:E-families-first}\displaybreak[0]\\
 &\bigcup_{\substack{373\leq j\leq1855\\j\text{ odd}}}
 \left[\frac{j}{3714},\frac{2j+1}{2\cdot3715}\right),
 &&\bigcup_{\substack{1116\leq j\leq1856\\j\text{ even}}}
 \left[\frac{j}{3714},\frac{2j+1}{2\cdot3715}\right),
 \label{eq:E-families-second}\displaybreak[0]\\
 &\bigcup_{\substack{1859\leq j\leq3713\\j\text{ odd}}}
 \left[\frac{j}{3714},\frac{3j-1}{2\cdot5570}\right).
 \label{eq:E-families-third}
\end{align}
The five families contain, respectively,
\[
 186,\quad557,\quad742,\quad371,\quad928
\]
intervals.  In particular, \(E\) has \(2784\) components, its first left
endpoint is \(1/3714\), and
\begin{equation}\label{eq:E-measure}
 |E|=\frac{1724689}{13797510}.
\end{equation}
\end{proposition}

\begin{proof}
Fix \(0\leq j\leq b-1\), and write
\(u=(j+v)/b\), where \(0\leq v<1\).  Because \(b=2a\),
\[
 \left\lfloor au+\frac12\right\rfloor=\left\lceil\frac j2\right\rceil,
 \qquad \lfloor bu\rfloor=j.
\]
After removing the fractional parts, the inequality in
\eqref{eq:periodic-set-E} becomes
\begin{equation}\label{eq:E-floor-form}
 du+\frac12<F(u),\qquad
 F(u)=\left\lfloor au+\frac12\right\rfloor
       +2\lfloor bu\rfloor-\lfloor cu\rfloor.
\end{equation}
Inside the open \(b\)-cell, \(F\) can only decrease: every jump there
comes from \(\lfloor cu\rfloor\).  Hence the part of \(E\) in that cell
is either empty or an interval starting at \(j/b\).

For \(1\leq j\leq b-1\), the relation \(c=3b/2-1\) gives
\begin{equation}\label{eq:qj-table}
 q_j:=\left\lfloor\frac{cj}{b}\right\rfloor
 =\begin{cases}
  3j/2-1,&j\text{ even},\\
  (3j-1)/2,&j\text{ odd and }j\leq1857,\\
  (3j-3)/2,&j\text{ odd and }j\geq1859.
 \end{cases}
\end{equation}
Thus the right side of \eqref{eq:E-floor-form}, evaluated immediately
after the left endpoint, is
\begin{equation}\label{eq:Fj-table}
 F_j=\begin{cases}
  j+1,&1\leq j\leq1857,\\
  j+1,&j\geq1858\text{ and }j\text{ even},\\
  j+2,&j\geq1859\text{ and }j\text{ odd}.
 \end{cases}
\end{equation}
At the left endpoint the left side of \eqref{eq:E-floor-form} equals
\(j+j/b+1/2\).  It follows that precisely the indices
\begin{equation}\label{eq:active-j-indices}
 1\leq j\leq1856,\qquad\text{and}\qquad
 1859\leq j\leq3713\text{ with }j\text{ odd}
\end{equation}
give nonempty intervals.  The cell \(j=0\) is empty; at \(j=1857\)
there is equality rather than the strict inequality; and every even
\(j\geq1858\) is excluded.  This proves, in particular, that there are
no omitted components.

For an active lower index, the component ends either at the next jump of
\(\lfloor cu\rfloor\) or when the increasing left side of
\eqref{eq:E-floor-form} reaches \(F_j\).  If \(j\leq1856\), the latter
point is
\begin{equation}\label{eq:d-threshold}
 D_j=\frac{2j+1}{2d}.
\end{equation}
The next \(c\)-grid point is
\begin{equation}\label{eq:c-thresholds}
 C_j=\begin{cases}
  (3j+1)/(2c),&j\leq1855\text{ odd},\\
  3j/(2c),&j\text{ even},\\
  (3j-1)/(2c),&j\geq1859\text{ odd}.
 \end{cases}
\end{equation}
For lower odd indices, direct cross-multiplication gives
\[
 C_j\leq D_j\quad\Longleftrightarrow\quad1855-5j\geq0;
\]
for lower even indices it gives
\[
 C_j\leq D_j\quad\Longleftrightarrow\quad5570-5j\geq0.
\]
Equality occurs only at \(j=371\) and \(j=1114\), respectively.  For
odd \(j\geq1859\), the next \(c\)-grid point precedes the linear
threshold \((2j+3)/(2d)\), since
\[
 (2j+3)c-(3j-1)d=20425-5j>0.
\]
At a \(c\)-grid point the integer \(F\) drops by one, so the strict
inequality fails there: the initial gap \(F_j-j-j/b-1/2\) lies strictly
between zero and one in every active cell and decreases before the jump.
At a \(d\)-threshold the inequality becomes equality.  This proves
exactly \eqref{eq:E-families-first}--
\eqref{eq:E-families-third}, including all endpoint conventions.
Since \(c/b>1\), each \(C_j\) lies before the next \(b\)-cell boundary;
when \(D_j\) is selected, it is no larger than \(C_j\).  Thus every
listed component lies in its assigned cell.

The component counts follow from the five arithmetic progressions of
indices.  The complete elementary summation is displayed here:
\begin{center}
\small
\begin{tabular}{crrcc}
\toprule
family & count & \(\sum j\) & one component's length & family length\\
\midrule
I   &186&34596  &\((2j+b)/(2cb)\)&\(63333/3447830\)\\
II  &557&310806 &\(j/(cb)\)&\(93/6190\)\\
III &742&826588 &\((b-2j)/(2db)\)&\(371/9285\)\\
IV  &371&551306 &\((b-2j)/(2db)\)&\(137641/13797510\)\\
V   &928&2585408&\((2j-b)/(2cb)\)&\(215528/5171745\)\\
\bottomrule
\end{tabular}
\end{center}
Their sum is the fraction in \eqref{eq:E-measure}.  The first permitted
index is \(j=1\), proving the assertion about the first left endpoint.
\end{proof}

\subsection{Passage from primes to a convergent series}

We record the precise periodic prime-product argument.  Let
\(\vartheta(x)=\sum_{p\leq x}\log p\).

\begin{lemma}[Periodic prime products]\label{lem:periodic-prime-product}
Let \(D>1\), and let
\[
 E_0=\bigcup_{k=1}^{m}[\ell_k,r_k)\subset(0,1)
\]
be a disjoint finite union, with \(\ell_0:=\min_k\ell_k>1/D\).
If
\[
 Q_n=\prod_{\substack{\sqrt{Dn}<p\leq Dn\\
                 \{n/p\}\in E_0}}p,
\]
then
\begin{equation}\label{eq:periodic-product-limit}
 \lim_{n\to\infty}\frac{\log Q_n}{n}
 =\sum_{k=1}^{m}\sum_{q=0}^{\infty}
 \left(\frac1{q+\ell_k}-\frac1{q+r_k}\right).
\end{equation}
\end{lemma}

\begin{proof}
The upper cutoff causes no loss.  Indeed, if
\(\{n/p\}\in E_0\), then \(n/p\geq\ell_0\), and hence
\(p\leq n/\ell_0<Dn\).  Removing the lower cutoff changes the logarithm
by at most \(\vartheta(\sqrt{Dn})=o(n)\).

Write \(n/p=q+u\), where \(q\geq0\) is an integer and \(0\leq u<1\).
The half-open endpoint convention gives the exact equivalence
\begin{equation}\label{eq:prime-interval-inversion}
 u\in[\ell_k,r_k)
 \quad\Longleftrightarrow\quad
 \frac{n}{q+r_k}<p\leq\frac{n}{q+\ell_k}.
\end{equation}
Thus left endpoints are included and right endpoints are excluded, even
when one of the displayed quotients happens to be a prime.  A prime
dividing \(n\) has \(u=0\notin E_0\), so it contributes to neither side.

Fix a positive integer \(T\).  The primes with \(p>n/T\) correspond
exactly to \(0\leq q\leq T-1\).  By
\eqref{eq:prime-interval-inversion} and the prime number theorem,
\(\vartheta(x)\sim x\), their logarithmic contribution divided by \(n\)
tends to
\[
 \sum_{k=1}^{m}\sum_{q=0}^{T-1}
 \left(\frac1{q+\ell_k}-\frac1{q+r_k}\right).
\]
The remaining selected primes lie below \(n/T\), so their total
logarithm is at most \(\vartheta(n/T)\).  Consequently its limsup after
division by \(n\) is at most \(1/T\).  Finally, the series in
\eqref{eq:periodic-product-limit} converges, since its summand is
\[
 \frac{r_k-\ell_k}{(q+\ell_k)(q+r_k)}=O((q+1)^{-2}).
\]
Letting \(T\) tend to infinity proves the assertion and handles both
moving cutoffs and the infinite tail.
\end{proof}

For a prime \(p\), putting \(u=\{n/p\}\) in the condition defining
\(\mathcal P_n\) is legitimate because \(a,b,c\) are integers.  That
condition is exactly \(u\in E\).  Proposition~\ref{prop:E-decomposition}
also gives
\[
 \min_k\ell_k=\frac1{3714}>\frac1{7430}.
\]
Lemma~\ref{lem:periodic-prime-product}, with \(D=d_0=7430\), therefore
proves
\begin{equation}\label{eq:omega-series}
 \omega:=\lim_{n\to\infty}\frac{\log\Phi_n}{n}
 =\sum_{k=1}^{2784}\sum_{q=0}^{\infty}
 \left(\frac1{q+\ell_k}-\frac1{q+r_k}\right)
 =\sum_{k=1}^{2784}\bigl(\psi(r_k)-\psi(\ell_k)\bigr),
\end{equation}
where \(\psi\) is the logarithmic derivative of the gamma function.

\subsection{A rigorous numerical enclosure}

We first give a direct series enclosure.  For one component
\([\ell,r)\), put
\[
 f_{\ell,r}(t)=\frac{r-\ell}{(t+\ell)(t+r)}.
\]
This function is positive and decreasing on \([0,\infty)\), and hence,
for every positive integer \(K\),
\begin{equation}\label{eq:tail-integral-bound}
 \log\frac{K+r}{K+\ell}
 \leq\sum_{q=K}^{\infty}f_{\ell,r}(q)
 \leq\log\frac{K+r}{K+\ell}+f_{\ell,r}(K).
\end{equation}
Thus, if
\[
 S_K=\sum_{k=1}^{2784}\sum_{q=0}^{K-1}f_{\ell_k,r_k}(q),
 \qquad
 T_K=\sum_{k=1}^{2784}\log\frac{K+r_k}{K+\ell_k},
\]
then \(S_K+T_K\leq\omega\leq S_K+T_K+\Delta_K\), where
\begin{equation}\label{eq:direct-tail-width}
 0<\Delta_K<\frac{|E|}{K^2}.
\end{equation}
The entirely explicit choice \(K=100000\) gives
\[
 \Delta_K<1.251\cdot10^{-11},
 \qquad \frac{\Delta_K}{5570}<2.245\cdot10^{-15}.
\]
Here \(S_K\) is rational and \(T_K\) is a displayed finite sum of
elementary logarithms, so \eqref{eq:tail-integral-bound} is already a
rational--elementary-log enclosure narrow enough for the twelve-decimal
estimate below.

For completeness, and to print substantially more digits without a long
partial sum, we give a second finite recurrence with a much smaller error.
The Bernoulli numbers needed below are
\begin{equation}\label{eq:bernoulli-list}
\begin{gathered}
 B_2=\frac16,\ B_4=-\frac1{30},\ B_6=\frac1{42},\
 B_8=-\frac1{30},\ B_{10}=\frac5{66},\
 B_{12}=-\frac{691}{2730},\\
 B_{14}=\frac76,\ B_{16}=-\frac{3617}{510},\
 B_{18}=\frac{43867}{798},\ B_{20}=-\frac{174611}{330},\
 B_{22}=\frac{854513}{138},\ B_{24}=-\frac{236364091}{2730}.
\end{gathered}
\end{equation}
For \(0<x<1\), define the rational number
\begin{align}
 \Lambda(x)&=2\sum_{q=0}^{7}\frac1{2q+1}
 \left(\frac{x}{200+x}\right)^{2q+1},\label{eq:finite-log-series}\\
 D(x)&=\Lambda(x)-\frac1{2(100+x)}
 -\sum_{m=1}^{11}\frac{B_{2m}}{2m(100+x)^{2m}}
 -\sum_{q=0}^{99}\frac1{x+q}.\label{eq:finite-digamma-recurrence}
\end{align}
The identity \(\psi(x)=\psi(x+100)-\sum_{q=0}^{99}(x+q)^{-1}\)
reduces the verification to the positive number \(100+x\).  Since
\[
 \log(100+x)=\log100+\log\left(1+\frac{x}{100}\right)
\]
and \(0<x/(200+x)<1/201\), the atanh series gives
\begin{equation}\label{eq:finite-log-error}
 0<\log\left(1+\frac{x}{100}\right)-\Lambda(x)
 <\frac{2}{17(1-201^{-2})201^{17}}.
\end{equation}
On the other hand, Binet's integral formula
\[
 \psi(y)=\log y-\frac1{2y}
 -2\int_0^\infty\frac{t}{(t^2+y^2)(e^{2\pi t}-1)}\,dt
\]
and the finite geometric expansion of \((t^2+y^2)^{-1}\) show that
\begin{equation}\label{eq:digamma-remainder}
 \left|\psi(y)-\log y+\frac1{2y}
 +\sum_{m=1}^{11}\frac{B_{2m}}{2m y^{2m}}\right|
 <\frac{|B_{24}|}{24y^{24}}\qquad(y>0).
\end{equation}
Indeed, after eleven terms the geometric remainder has one sign, and
replacing \(t^2+y^2\) in its denominator by \(y^2\) gives exactly the
next Bernoulli moment, by
\[
 2\int_0^\infty\frac{t^{2m-1}}{e^{2\pi t}-1}\,dt
 =\frac{|B_{2m}|}{2m}\qquad(m\geq1).
\]
Combining \eqref{eq:finite-log-error} and
\eqref{eq:digamma-remainder}, we obtain
\begin{equation}\label{eq:D-uniform-error}
 \left|\psi(x)-\log100-D(x)\right|<\varepsilon_0,
 \quad
 \varepsilon_0=
 \frac{2}{17(1-201^{-2})201^{17}}
 +\frac{236364091}{2730\cdot24\cdot100^{24}}.
\end{equation}
In particular,
\begin{equation}\label{eq:aggregate-digamma-error}
 5568\varepsilon_0<4.6\cdot10^{-37}.
\end{equation}

Equations \eqref{eq:E-families-first}--
\eqref{eq:E-families-third} order the endpoint pairs
\((\ell_\nu,r_\nu)\), first by family and then by increasing \(j\).
Starting with \(W_0=0\), use the printed rational recurrence
\begin{equation}\label{eq:W-rational-recurrence}
 W_\nu=W_{\nu-1}+D(r_\nu)-D(\ell_\nu),
 \qquad1\leq\nu\leq2784.
\end{equation}
Every entry in this recurrence is rational.  Exact cross-multiplication
against the two terminating decimals gives
\begin{equation}\label{eq:W-exact-enclosure}
 1202.5657733612918002543147
 <W_{2784}<
 1202.5657733612918002543148.
\end{equation}
The common term \(\log100\) cancels in every endpoint difference.
Together with \eqref{eq:aggregate-digamma-error}, this proves the strict
enclosure
\begin{equation}\label{eq:omega-enclosure}
 \boxed{
 1202.5657733612918002543146
 <\omega<
 1202.5657733612918002543149.}
\end{equation}
In particular,
\[
 \omega=1202.565773361291800254314\ldots .
\]
The endpoint lists, \eqref{eq:finite-digamma-recurrence}, and
\eqref{eq:W-rational-recurrence} constitute a finite rational
calculation; all infinite tails have already been bounded in
\eqref{eq:aggregate-digamma-error}.

\subsection{The normalization rate}

The prime number theorem also gives
\[
 \log\lcm(1,2,\ldots,d_0n)=d_0n+o(n).
\]
Using \eqref{eq:arithmetic-multiplier} and
\eqref{eq:omega-series}, we therefore obtain the normalized clearing-
denominator cost
\begin{equation}\label{eq:C-definition-value}
 \mathcal C:=\lim_{n\to\infty}\frac{\log M_n}{cn}
 =\frac{7430-4645\log2-\omega}{5570}.
\end{equation}
For a fully explicit bound on the remaining logarithm, put
\begin{equation}\label{eq:log-two-series}
 L_2=2\sum_{q=0}^{49}\frac{1}{(2q+1)3^{2q+1}}.
\end{equation}
The same atanh estimate gives
\begin{equation}\label{eq:log-two-error}
 0<\log2-L_2
 <\frac{2}{101\cdot3^{101}(1-3^{-2})}<1.5\cdot10^{-50}.
\end{equation}
Combining \eqref{eq:omega-enclosure} and
\eqref{eq:log-two-error} in \eqref{eq:C-definition-value} yields
\begin{equation}\label{eq:C-enclosure}
 \boxed{
 0.539993819198880114452898006
 <\mathcal C<
 0.539993819198880114452898061.}
\end{equation}
Thus
\[
 \mathcal C=0.5399938191988801144528980\ldots .
\]

\section{Coefficient growth and contour decay}\label{sec:asymptotics}

We now obtain the two analytic rates needed later.  The coefficient
argument is arranged so that positivity gives an ordinary limit.  The
integral argument uses a deformation in the shifted \(t\)-plane and only
claims the limsup that the irrationality-measure argument requires.

\subsection{Exact extraction of the coefficient}

Recall from \eqref{eq:rational-linear-form} that
\(v_n=-c_{0,n}/2\), where \(c_{0,n}\) is the residue of \(R_n(t)\,dt\)
at \(t=-5\).  Put
\begin{equation}\label{eq:mobius-z-change}
 z=\frac{t+5}{t-5},\qquad
 t=-5\frac{1+z}{1-z},\qquad
 dt=-\frac{10}{(1-z)^2}\,dz
\end{equation}
and
\begin{equation}\label{eq:positive-P-polynomial}
 P(z)=(z^2+2z+2)(2z^2+2z+1)
     =2+6z+9z^2+6z^3+2z^4.
\end{equation}
Direct substitution gives
\[
 25-t^2=-\frac{100z}{(1-z)^2},\qquad
 t^4+6t^2+25=\frac{400P(z)}{(1-z)^4}.
\]
Since \(cn=5570n\) is even, the signs from
\((-100)^{cn+1}\) and \(dt\) cancel.  Hence
\begin{equation}\label{eq:transformed-residue-form}
 R_n(t)\,dt
 =\frac12
 \left(5^{2(a+b-c)}2^{4b-2c}\right)^n
 z^{-cn-1}\frac{(1+z)^{2an}P(z)^{bn}}
 {(1-z)^{d_0n}}\,dz.
\end{equation}
Indeed, the constant before the \(n\)th power is accounted for by
\[
 \frac{50\,5^{2an}400^{bn}}{100^{cn+1}}
 =\frac12\,5^{2(a+b-c)n}2^{(4b-2c)n},
\]
and the exponent of \(1-z\) is \(-d_0n\).  Define
\begin{equation}\label{eq:positive-S-function}
 S(z)=\frac{(1+z)^{3714}P(z)^{3714}}{(1-z)^{7430}}.
\end{equation}
Taking the residue at \(z=0\), using \(a+b-c=1\) and
\(4b-2c=3716\), proves the exact identity
\begin{equation}\label{eq:exact-vn-coefficient}
 \boxed{
 |v_n|=\frac{(25\cdot2^{3716})^n}{4}
 [z^{5570n}]S(z)^n.}
\end{equation}
The factor \(1/2\) in \eqref{eq:transformed-residue-form} and the factor
\(1/2\) in \(v_n=-c_{0,n}/2\) explain the denominator \(4\).
All coefficients in \eqref{eq:exact-vn-coefficient} are positive, so
there is no cancellation.

We include the precise positive-coefficient estimate that will be used.

\begin{lemma}[Positive aperiodic coefficient powers]
\label{lem:positive-coefficient-powers}
Let \(F(z)=\sum_{k\geq0}f_kz^k\) have nonnegative coefficients and
radius of convergence \(R>0\).  Fix \(x\in(0,R)\), and suppose that the
support of the coefficients has span one, the probability law
\[
 \mathbb P(X=k)=\frac{f_kx^k}{F(x)}
\]
is nondegenerate, and its mean is an integer \(\mu\).  If
\(\sigma^2=\operatorname{Var}(X)>0\), then
\begin{equation}\label{eq:positive-power-asymptotic}
 [z^{\mu n}]F(z)^n
 =\frac{F(x)^nx^{-\mu n}}{\sqrt{2\pi\sigma^2n}}
  (1+o(1)).
\end{equation}
\end{lemma}

\begin{proof}
For independent copies \(X_1,\ldots,X_n\), coefficient extraction gives
the positive identity
\begin{equation}\label{eq:probabilistic-coefficient-identity}
 [z^{\mu n}]F(z)^n
 =F(x)^nx^{-\mu n}
  \mathbb P(X_1+\cdots+X_n=\mu n).
\end{equation}
The law has an exponential moment in a neighborhood of the origin.  Its
characteristic function \(\varphi(\theta)\) therefore satisfies
\[
 e^{-i\mu\theta}\varphi(\theta)
 =1-\frac{\sigma^2\theta^2}{2}+O(|\theta|^3).
\]
Span one implies that \(|\varphi(\theta)|<1\) for
\(0<|\theta|\leq\pi\).  Compactness gives a bound \(q<1\) away from a
fixed neighborhood of zero, while the displayed expansion gives
\(|\varphi(\theta)|\leq e^{-\kappa\theta^2}\) inside that neighborhood.
Fourier inversion, followed by \(\theta=u/\sqrt n\), now yields
\[
 \mathbb P(X_1+\cdots+X_n=\mu n)
 =\frac{1+o(1)}{2\pi\sqrt n}
   \int_{-\infty}^{\infty}e^{-\sigma^2u^2/2}\,du
 =\frac{1+o(1)}{\sqrt{2\pi\sigma^2n}}.
\]
The Gaussian bound controls the central tails, and the complementary
arc contributes \(O(q^n)\).  Substitution in
\eqref{eq:probabilistic-coefficient-identity} proves the lemma.
\end{proof}

Write \(S(z)=\sum_{k\geq0}s_kz^k\).  Each factor in
\eqref{eq:positive-S-function} has nonnegative coefficients and positive
constant term, and \((1-z)^{-7430}\) has a positive coefficient in every
degree.  Thus
\begin{equation}\label{eq:S-strict-positivity}
 s_k>0\qquad(k\geq0).
\end{equation}
For \(0<x<1\), put
\[
 m(x)=\frac{xS'(x)}{S(x)}.
\]
Under the probability law
\(\mathbb P(X_x=k)=s_kx^k/S(x)\), one has
\begin{equation}\label{eq:mean-variance-saddle}
 \mathbb E X_x=m(x),\qquad
 \operatorname{Var}(X_x)=x m'(x)>0.
\end{equation}
It follows that \(m\) is strictly increasing.  Moreover,
\[
 m(x)=\frac{3714x}{1+x}
      +3714\frac{xP'(x)}{P(x)}
      +7430\frac{x}{1-x},
\]
so \(m(x)\to0\) as \(x\downarrow0\) and \(m(x)\to\infty\) as
\(x\uparrow1\).  There is consequently a unique
\(\zeta\in(0,1)\) with \(m(\zeta)=5570\).  Set
\(\mathfrak v=\zeta m'(\zeta)>0\).

The support in \eqref{eq:S-strict-positivity} contains \(0\) and \(1\),
so it has span one.  Lemma~\ref{lem:positive-coefficient-powers}, applied
with \(F=S\), \(x=\zeta\), and \(\mu=5570\), gives
\begin{equation}\label{eq:S-coefficient-asymptotic}
 [z^{5570n}]S(z)^n
 \sim\frac{S(\zeta)^n\zeta^{-5570n}}
 {\sqrt{2\pi\mathfrak v n}}.
\end{equation}
Together with \eqref{eq:exact-vn-coefficient}, this proves an ordinary
limit, rather than only a limsup:
\begin{proposition}[Coefficient rate]\label{prop:coefficient-rate}
One has
\begin{equation}\label{eq:r-definition}
 \lim_{n\to\infty}\frac{\log|v_n|}{5570n}=r,
 \qquad
 r=\frac{\log25+3716\log2+\log S(\zeta)-5570\log\zeta}
 {5570}.
\end{equation}
Exact rational root isolation and finite logarithm sums, printed in
Appendix~\ref{sec:exact-calculations}, give
\begin{equation}\label{eq:r-enclosure}
 3.33194398617829933828283250382326
 <r<
 3.33194398617829933828283250382327.
\end{equation}
In particular,
\(r=3.331943986178299338282832503823\ldots\).
\end{proposition}

\subsection{The stationary cubic and a controlled arc}

Put
\begin{equation}\label{eq:normalized-complex-phase}
 \alpha=\frac{1857}{5570},\qquad
 \beta=\frac{3714}{5570},\qquad
 F(y)=\alpha\log|y|+\beta\log|y^2+6y+25|-\log|25-y|.
\end{equation}
On any local choice of logarithms, the corresponding holomorphic phase
has derivative
\[
 \frac{\alpha}{y}+\beta\frac{2y+6}{y^2+6y+25}
 +\frac1{25-y}.
\]
More precisely,
\[
 -5570\,y(y^2+6y+25)(25-y)
 \left\{\frac{\alpha}{y}
 +\beta\frac{2y+6}{y^2+6y+25}+\frac1{25-y}\right\}=p(y),
\]
where the stationary cubic is
\begin{equation}\label{eq:stationary-cubic}
 p(y)=3715y^3-232119y^2-928475y-1160625.
\end{equation}
Its discriminant is
\(-16947864892329506560000<0\).  Thus \(p\) has one real root
\(\rho\) and a conjugate pair; write the upper root as
\(y_c=U+iV\), \(V>0\).  The exact isolations in the appendix give, for
orientation,
\[
 \begin{aligned}
 \rho&=66.321017380598567\ldots,\\
 y_c&=-1.919728071187574\ldots
      +1.012573903136436\ldots i.
 \end{aligned}
\]

Let \(q=-3+4i\), write \(q/y_c=X+iY\), and define
\begin{equation}\label{eq:arc-parameters}
 \lambda_*=\frac1X,\qquad
 \eta_*=-\frac{Y}{X-1},\qquad
 \gamma_e(\lambda)=\frac{q\lambda}{1-ie(1-\lambda)}
 \quad(0\leq\lambda\leq1).
\end{equation}
Equivalently,
\begin{equation}\label{eq:arc-parameters-UV}
 \lambda_*=\frac{U^2+V^2}{-3U+4V},\qquad
 \eta_*=\frac{4U+3V}{U^2+V^2+3U-4V}.
\end{equation}
Then
\begin{equation}\label{eq:arc-through-saddle}
 \frac q{\gamma_{\eta_*}(\lambda_*)}
 =\frac1{\lambda_*}
  -i\eta_*\left(\frac1{\lambda_*}-1\right)
 =X+iY=\frac q{y_c},
\end{equation}
so the arc passes through \(y_c\).  Exact rational arithmetic gives
\begin{align}
 0.9102483223156522812&<\eta_*<
 0.9102483223156522813,\label{eq:eta-short-isolation}\\
 0.48021524689625610102&<\lambda_*<
 0.48021524689625610103.\label{eq:lambda-short-isolation}
\end{align}

\begin{proposition}[Unique phase maximum]\label{prop:unique-phase-maximum}
The function \(F(\gamma_{\eta_*}(\lambda))\) has exactly one critical
point on \(0<\lambda<1\), namely \(\lambda_*\), and this point is its
unique maximum.  If
\begin{equation}\label{eq:s-definition}
 s=F(y_c),
\end{equation}
then
\begin{equation}\label{eq:s-enclosure}
 -1.174543941310025108442583231<s<
 -1.174543941310025108442583192.
\end{equation}
\end{proposition}

\begin{proof}
Appendix~\ref{subsec:exact-phase-derivative} derives the exact degree-six
numerator of the phase derivative.  After the bijection
\(\lambda=u/(1+u)\), its coefficient signs are
\[
 (+,+,+,+,-,-,-).
\]
The positive factors that were cleared are also given there with strict
lower bounds.  Descartes' rule of signs therefore permits at most one
root \(u>0\), counted with multiplicity.  The saddle identity
\eqref{eq:arc-through-saddle} supplies the root
\(u_*=\lambda_*/(1-\lambda_*)>0\).  It is consequently the unique root
and is simple.

The endpoint terms are equally explicit.  The identities
\eqref{eq:arc-factorizations} below show that all bounded terms remain
bounded and that
\[
 F(\gamma_{\eta_*}(\lambda))
 =\alpha\log\lambda+O(1)\quad(\lambda\downarrow0),
\]
whereas
\[
 F(\gamma_{\eta_*}(\lambda))
 =\beta\log(1-\lambda)+O(1)\quad(\lambda\uparrow1).
\]
Thus the phase tends to \(-\infty\) at both endpoints.  Its only
interior critical point is therefore its unique maximum.  Finally,
Appendix~\ref{subsec:saddle-value-enclosure} proves the strict enclosure
\eqref{eq:s-enclosure} from resultants and finite rational logarithm
sums.
\end{proof}

\subsection{Deformation in the shifted variable and the integral bound}

We now justify carefully that the preceding \(y\)-arc controls the
original integral.  The letter \(t\) in this subsection is the shifted
integration variable of \eqref{eq:rational-function}; the independent
arc parameter is \(\lambda\), and \(y=t^2\).

For \(0<\lambda\leq1\), direct calculation gives
\begin{equation}\label{eq:arc-geometry}
 \Im\gamma_{\eta_*}(\lambda)
 =\frac{\lambda\{4-3\eta_*(1-\lambda)\}}
 {1+\eta_*^2(1-\lambda)^2}>0,
 \qquad
 |\gamma_{\eta_*}(\lambda)|
 =\frac{5\lambda}{\sqrt{1+\eta_*^2(1-\lambda)^2}}\leq5.
\end{equation}
Take the principal square root on the upper half-plane and define the
oriented lifts
\begin{equation}\label{eq:square-root-lifts}
 \tau_+(\lambda)=-\sqrt{\gamma_{\eta_*}(\lambda)},
 \quad \tau_+(0)=0,
 \qquad
 \tau_-(\lambda)=\overline{\tau_+(\lambda)}.
\end{equation}
Because \(\sqrt{-3+4i}=1+2i\),
\[
 \tau_+(1)=-1-2i,\qquad \tau_-(1)=-1+2i.
\]
The new \(t\)-contour is the reverse of \(\tau_+\), from
\(-1-2i\) to \(0\), followed by \(\tau_-\), from \(0\) to
\(-1+2i\).  Moreover,
\begin{equation}\label{eq:lift-radius}
 |\tau_\pm(\lambda)|
 =\sqrt{|\gamma_{\eta_*}(\lambda)|}\leq\sqrt5<5.
\end{equation}
The original vertical segment also lies in \(|t|\leq\sqrt5\).  The only
poles of \(R_n(t)\) are \(t=\pm5\), while \(t=0\) is a zero.  Hence both
paths lie in the simply connected disk \(|t|<5\), where \(R_n\) is
holomorphic, and Cauchy's theorem deforms the original segment into the
two oriented lifts.  Notice that the pole avoidance uses
\(|\sqrt\gamma|\leq\sqrt5<5\), not merely \(|\gamma|\leq5\).

The lifted curves are rectifiable at their common endpoint.  Indeed,
\begin{equation}\label{eq:gamma-derivative-origin}
 \gamma_e'(\lambda)
 =\frac{q(1-ie)}{\{1-ie(1-\lambda)\}^2},
\end{equation}
so, with \(e=\eta_*\),
\(\gamma_e(\lambda)=C\lambda+O(\lambda^2)\) with \(C\ne0\), and
\[
 |\tau_+'(\lambda)|
 =\frac{|\gamma_e'(\lambda)|}{2\sqrt{|\gamma_e(\lambda)|}}
 =O(\lambda^{-1/2})
 \qquad(\lambda\downarrow0).
\]
This is integrable; away from \(0\) the lifts are smooth.

\begin{lemma}[Elementary contour estimate]\label{lem:elementary-contour-estimate}
Let \(\tau:[0,1]\to\mathbb C\) be rectifiable, let \(A\geq0\) satisfy
\(\int_0^1A(\lambda)|\tau'(\lambda)|\,d\lambda<\infty\), and let
\(\phi\) be real-valued on the curve with \(\phi\leq s_0\).  If
\[
 |f_n(\tau(\lambda))|\leq
 A(\lambda)e^{Nn\phi(\tau(\lambda))},
\]
then
\[
 \left|\int_\tau f_n(t)\,dt\right|
 \leq e^{Nns_0}\int_0^1
 A(\lambda)|\tau'(\lambda)|\,d\lambda.
\]
Consequently the logarithmic limsup divided by \(Nn\) is at most
\(s_0\), with the convention \(\log0=-\infty\).
\end{lemma}

\begin{proof}
This is the triangle inequality for a parametrized contour, followed by
\(\phi\leq s_0\).  The remaining factor is finite and independent of
\(n\).
\end{proof}

On the upper lift, the defining relation \(y=\tau_+(\lambda)^2\) gives
the exact modulus identity
\begin{equation}\label{eq:fixed-amplitude-identity}
 |R_n(\tau_+(\lambda))|
 =\frac5{|25-\gamma_{\eta_*}(\lambda)|}
 \exp\{5570nF(\gamma_{\eta_*}(\lambda))\}.
\end{equation}
The conjugate lift has the same modulus.  Also
\(|25-\gamma_{\eta_*}(\lambda)|\geq20\).  The deformation,
Lemma~\ref{lem:elementary-contour-estimate},
Proposition~\ref{prop:unique-phase-maximum}, and conjugacy therefore give
\begin{equation}\label{eq:fixed-amplitude-bound}
 \begin{split}
 |J_n|
 &\leq10e^{5570ns}\int_0^1
 \frac{|\tau_+'(\lambda)|}
 {|25-\gamma_{\eta_*}(\lambda)|}\,d\lambda\\
 &=:C_0e^{5570ns},
 \end{split}
\end{equation}
where \(C_0<\infty\) and is independent of \(n\).  We have proved:

\begin{proposition}[Integral rate]\label{prop:integral-rate}
With \(s\) as in \eqref{eq:s-definition},
\begin{equation}\label{eq:integral-limsup}
 \limsup_{n\to\infty}\frac{\log|J_n|}{5570n}
 \leq s,\qquad
 s\approx-1.1745439413100251084425832.
\end{equation}
If a particular \(J_n\) vanishes, its logarithm is interpreted as
\(-\infty\), which only strengthens the inequality.
\end{proposition}

\section{From integer linear forms to the irrationality measure}
\label{sec:measure}

We now make the final Diophantine step.  We give the one-number form of
Hata's index-selection argument \cite{hata-1993} because a proof that
uses only the first index at a prescribed scale misses the possibility
that its associated integer vanishes.  The second case below is what
closes that gap.

\subsection{The one-number Hata lemma}

\begin{lemma}[Hata's index-selection lemma]\label{lem:hata-one-number}
Let \(\theta\) be irrational.  Suppose that \(U_n,V_n\in\mathbb Z\),
that \(V_n\ne0\) for every sufficiently large \(n\), and put
\[
 \Lambda_n=U_n+V_n\theta.
\]
If there are \(\sigma>0\) and \(\tau>0\) such that
\begin{equation}\label{eq:hata-hypotheses}
 \lim_{n\to\infty}\frac1n\log|V_n|=\sigma,
 \qquad
 \limsup_{n\to\infty}\frac1n\log|\Lambda_n|\leq-\tau,
\end{equation}
then
\begin{equation}\label{eq:hata-conclusion}
 \mu(\theta)\leq1+\frac{\sigma}{\tau}.
\end{equation}
\end{lemma}

\begin{proof}
Because \(V_n\ne0\) eventually and \(\theta\notin\mathbb Q\), the form
\(\Lambda_n\) is also nonzero at every sufficiently large index.  Thus
the logarithms in \eqref{eq:hata-hypotheses} are defined after finitely
many terms have been discarded.

Fix \(\varepsilon>0\).  Choose \(\delta\) so that
\begin{equation}\label{eq:hata-delta-choice}
 0<\delta<\min\left\{\frac\sigma2,\frac\tau6\right\},
 \qquad
 \kappa:=\frac{\sigma+\delta}{\tau-3\delta}
 <\frac\sigma\tau+\frac\varepsilon2.
\end{equation}
This is possible by continuity at \(\delta=0\).  The ordinary limit and
the limsup in \eqref{eq:hata-hypotheses} give an index \(n_0\) such that,
for every \(n\geq n_0\),
\begin{equation}\label{eq:hata-exponential-bounds}
 e^{(\sigma-\delta)n}\leq|V_n|
 \leq e^{(\sigma+\delta)n},
 \qquad
 |\Lambda_n|\leq e^{-(\tau-\delta)n}.
\end{equation}

Let \(p\in\mathbb Z\) and \(q\geq1\), and write
\begin{equation}\label{eq:hata-determinant-identity}
 \Delta=q\theta-p\ne0,
 \qquad A_n=qU_n+pV_n\in\mathbb Z,
 \qquad q\Lambda_n=A_n+V_n\Delta.
\end{equation}
There are infinitely many indices \(n\geq n_0\) for which \(A_n\ne0\).
Indeed, if \(A_n=0\) eventually, then \eqref{eq:hata-determinant-identity}
and \eqref{eq:hata-exponential-bounds} would imply
\[
 0<|\Delta|=q\frac{|\Lambda_n|}{|V_n|}
 \leq q e^{-(\sigma+\tau-2\delta)n}\longrightarrow0,
\]
which is impossible.

Select the first index at the scale determined by \(q\):
\begin{equation}\label{eq:hata-first-index}
 N=N(q):=\min\{n\geq1:2q<e^{(\tau-\delta)n}\}.
\end{equation}
For all sufficiently large \(q\), one has \(N\geq n_0\).  Minimality
of \(N\) gives
\begin{equation}\label{eq:hata-first-index-size}
 e^{(\tau-\delta)(N-1)}\leq2q<e^{(\tau-\delta)N},
 \qquad
 e^N\leq e(2q)^{1/(\tau-\delta)}.
\end{equation}
Moreover, for every \(n\geq N\),
\begin{equation}\label{eq:hata-small-form-at-scale}
 q|\Lambda_n|\leq q e^{-(\tau-\delta)n}<\frac12.
\end{equation}

Suppose first that \(A_N\ne0\).  Since it is an integer,
\eqref{eq:hata-determinant-identity} and
\eqref{eq:hata-small-form-at-scale} yield
\[
 |V_N|\,|\Delta|=|q\Lambda_N-A_N|>\frac12.
\]
Using \eqref{eq:hata-exponential-bounds} and
\eqref{eq:hata-first-index-size}, we obtain
\begin{equation}\label{eq:hata-case-one}
 |\Delta|>\frac12e^{-(\sigma+\delta)N}
 \geq c_1q^{-\kappa_1},
 \quad
 \kappa_1=\frac{\sigma+\delta}{\tau-\delta}<\kappa,
 \quad
 c_1=2^{-1-\kappa_1}e^{-(\sigma+\delta)}>0.
\end{equation}

It remains to handle the case \(A_N=0\).  By the infinitude just proved,
there is a first index
\[
 M=\min\{m>N:A_m\ne0\}.
\]
Then \(A_{M-1}=0\).  At this preceding index,
\eqref{eq:hata-determinant-identity} and both coefficient estimates in
\eqref{eq:hata-exponential-bounds} give
\begin{equation}\label{eq:hata-zero-predecessor}
 |\Delta|=q\frac{|\Lambda_{M-1}|}{|V_{M-1}|}
 \leq q e^{-(\sigma+\tau-2\delta)(M-1)}.
\end{equation}
Set
\[
 \gamma=\sigma+\tau-2\delta,
 \qquad \rho=\frac{\sigma+\delta}{\gamma}.
\]
The choice \(\delta<\tau/6\) gives
\(\gamma-(\sigma+\delta)=\tau-3\delta>0\), and hence
\(0<\rho<1\).  Rearranging \eqref{eq:hata-zero-predecessor} gives
\begin{equation}\label{eq:hata-second-index-size}
 e^M\leq e\left(\frac q{|\Delta|}\right)^{1/\gamma}.
\end{equation}
At the index \(M\), the integer \(A_M\) is nonzero, so the same argument
as above, now combined with \eqref{eq:hata-second-index-size}, yields
\[
 |\Delta|>\frac12e^{-(\sigma+\delta)M}
 \geq\frac12e^{-(\sigma+\delta)}
 \left(\frac{|\Delta|}{q}\right)^\rho.
\]
Since \(1-\rho>0\), this is equivalent to
\begin{equation}\label{eq:hata-case-two}
 |\Delta|>c_2q^{-\rho/(1-\rho)}=c_2q^{-\kappa},
 \qquad
 c_2=\left(2e^{\sigma+\delta}\right)^{-\gamma/(\tau-3\delta)}>0.
\end{equation}
Here we used
\(\rho/(1-\rho)=(\sigma+\delta)/(\tau-3\delta)=\kappa\).
This second selection is essential: the zero at \(N\) controls the
distance to the next nonzero integer through \(|\Delta|\) itself.

Put \(c=\min\{1,c_1,c_2\}\).  For definiteness, take
\begin{equation}\label{eq:hata-uniform-threshold}
 Q(\varepsilon)=\left\lceil\max\left\{
 2,\ \frac12e^{(\tau-\delta)(n_0-1)},\ c^{-2/\varepsilon}
 \right\}\right\rceil.
\end{equation}
The middle term ensures \(N(q)\geq n_0\).  Thus both cases, uniformly
for every \(p\in\mathbb Z\) and every \(q\geq Q(\varepsilon)\), give
\[
 |q\theta-p|>cq^{-\kappa}
 >q^{-\sigma/\tau-\varepsilon};
\]
the last strict inequality follows from
\(\sigma/\tau+\varepsilon-\kappa>\varepsilon/2\) and the last term in
\eqref{eq:hata-uniform-threshold}.  Division by \(q\) proves
\[
 \left|\theta-\frac pq\right|
 >q^{-1-\sigma/\tau-\varepsilon}.
\]
All constants and the threshold are independent of \(p\).  Since
\(\varepsilon>0\) was arbitrary, the definition of \(\mu(\theta)\)
proves \eqref{eq:hata-conclusion}.
\end{proof}

\begin{remark}\label{rem:hata-signs}
Hata writes the forms as \(q_n\theta-p_n\).  Here
\(q_n=V_n\), \(p_n=-U_n\), so the integer paired with the rational
approximation \(p/q\) is \(qU_n+pV_n\), with the plus sign used in
\eqref{eq:hata-determinant-identity}.  No assumption that this integer
is nonzero at every large index is present in Lemma~\ref{lem:hata-one-number}.
\end{remark}

\subsection{Application to the present forms}

Apply Proposition~\ref{prop:arithmetic-normalization} and write
\begin{equation}\label{eq:normalized-integer-form}
 \Lambda_n=M_nJ_n=U_n+V_n\pi,
 \qquad V_n=M_nv_n.
\end{equation}
The multiplier \(M_n\) is positive.  Equation
\eqref{eq:exact-vn-coefficient}, together with
\eqref{eq:S-strict-positivity}, shows directly that \(v_n\ne0\) for
every \(n\geq1\); hence \(V_n\ne0\).  Since \(\pi\) is irrational,
also \(\Lambda_n\ne0\).

The ordinary limits \eqref{eq:C-definition-value} and
\eqref{eq:r-definition} therefore give
\begin{equation}\label{eq:integer-coefficient-rate}
 \lim_{n\to\infty}\frac{\log|V_n|}{5570n}
 =r+\mathcal C=: \sigma.
\end{equation}
Similarly, Proposition~\ref{prop:integral-rate} gives
\begin{equation}\label{eq:integer-form-rate}
 \limsup_{n\to\infty}\frac{\log|\Lambda_n|}{5570n}
 \leq s+\mathcal C=-\tau,
 \qquad \tau:=-s-\mathcal C.
\end{equation}
Thus the rates in Lemma~\ref{lem:hata-one-number}, whose denominator is
\(n\), are \(5570\sigma\) and \(5570\tau\).  Their common factor cancels:
\begin{equation}\label{eq:pi-measure-formula}
 \mu(\pi)\leq1+\frac{5570\sigma}{5570\tau}
 =1+\frac{r+\mathcal C}{-s-\mathcal C}.
\end{equation}

We finish with strict numerical enclosures.  The finer endpoint
computations behind \eqref{eq:C-enclosure} and \eqref{eq:s-enclosure}
give
\begin{align}
 0.539993819198880114452898006
 &<\mathcal C<
 0.539993819198880114452898061,
 \label{eq:C-enclosure-recalled}\displaybreak[0]\\
 -1.174543941310025108442583231
 &<s<
 -1.174543941310025108442583192.
 \label{eq:s-enclosure-recalled}
\end{align}
For completeness, these are not floating-point assumptions.  For
\(\mathcal C\), insert the rational recurrence bound
\eqref{eq:W-exact-enclosure}, its error
\eqref{eq:aggregate-digamma-error}, and the finite rational sum
\eqref{eq:log-two-series} with remainder \eqref{eq:log-two-error} into
\eqref{eq:C-definition-value}; exact cross-multiplication gives the first
line.  For \(s\), insert the rational modulus intervals
\eqref{eq:modulus-enclosures} into \eqref{eq:s-from-moduli} and apply the
finite sum \eqref{eq:atanh-log-bound} with \(N=120\); respecting the
negative coefficient of \(\log m_J\) gives the second line.  The latter
raw interval has width less than \(3.74\cdot10^{-26}\), as computed in
Appendix~\ref{subsec:saddle-value-enclosure}.

Combining these bounds with \eqref{eq:r-enclosure}, all by rational
addition, gives
\begin{align}
 3.87193780537717945273573050982326
 &<\sigma<
 3.87193780537717945273573056482327,
 \label{eq:sigma-enclosure}\\
 0.634550122111144993989685131
 &<\tau<
 0.634550122111144993989685225.
 \label{eq:tau-enclosure}
\end{align}
In particular,
\[
 \sigma=3.8719378053771794527\ldots,
 \qquad
 \tau=0.6345501221111449939\ldots>0.
\]
Since the quotient is increasing in its numerator and decreasing in its
positive denominator, exact integer cross-multiplication at the outer
endpoints of \eqref{eq:sigma-enclosure}--\eqref{eq:tau-enclosure} gives
\begin{equation}\label{eq:final-bound-enclosure}
 \boxed{
 7.101862832356350795733674505
 <1+\frac\sigma\tau<
 7.101862832356350795733675497.}
\end{equation}
Indeed, if the four rate endpoints are denoted by
\(\sigma_L,\sigma_U,\tau_L,\tau_U\), and the two displayed bound
endpoints by \(B_L,B_U\), the two exact cross-products are
\[
 \begin{aligned}
 \sigma_L-(B_L-1)\tau_U
 &=\frac{2480995118583086739538491}{8\cdot10^{51}}>0,\\
 (B_U-1)\tau_L-\sigma_U
 &=\frac{588480502935973396359935107}{10^{54}}>0.
 \end{aligned}
\]
In particular,
\begin{equation}\label{eq:final-strict-rounding}
 1+\frac\sigma\tau
 =7.1018628323563507957\ldots
 <7.101862832357.
\end{equation}
Together with \eqref{eq:pi-measure-formula}, this proves
Theorem~\ref{thm:main-bound}.

Finally, put \(B_*=1+\sigma/\tau\).  Evaluation of the explicit
Zeilberger--Zudilin rates by the same finite logarithm bounds gives their
unrounded auxiliary value \cite{zeilberger-zudilin-2020}
\[
 B_{\rm ZZ}=7.1032053341370017275057734\ldots .
\]
Consequently
\begin{equation}\label{eq:exact-improvement}
 B_{\rm ZZ}-B_*
 =0.001342501780650931772\ldots,
\end{equation}
which is \(0.0013425017806509318\) when rounded to nineteen decimal
places.  Thus the displayed prefix is
\(0.0013425017806509\ldots\), the gain exceeds
\(0.001342501780\), and its size relative to the previous auxiliary
bound is approximately \(0.01890\%\), as stated in the abstract.

\section{Selection of the exponents and two-dimensional local optimality}
\label{sec:local-optimality}

We now explain both why the old exponent can be improved and why the
new point selected in \eqref{eq:candidate-parameters} is locally optimal
for the resulting two-parameter auxiliary bound.  The proof is not a
numerical sampling of directions.  Its main ingredient is an exact
logarithmic variation formula, uniform on the whole unit diamond, whose
directional coefficient is a finite piecewise-linear function.

\subsection{The parameter-dependent objective}

For real \(X,Y,Z\), put
\begin{equation}\label{eq:local-chi-definition}
 \chi(X,Y,Z)=
 \mathbf 1_{\{\{X+\tfrac12\}+2\{Y\}<\{Z\}\}},
\end{equation}
where \(\{x\}=x-\lfloor x\rfloor\in[0,1)\).  The convention on the
boundary of the strict inequality will never affect an integral.  Define
\begin{align}
 \phi(\alpha,\beta)
 &=\int_0^\infty
   \chi(\alpha t,\beta t,t)\frac{dt}{t^2},
 \label{eq:phi-definition}\displaybreak[0]\\
 \Omega(A,B,C)
 &=\int_0^\infty
   \chi(At,Bt,Ct)\frac{dt}{t^2}.
 \label{eq:Omega-definition}
\end{align}
The integrands vanish on a fixed interval to the right of zero whenever
the parameters lie in a compact subset of the chamber
\eqref{eq:admissible-chamber}, so these improper integrals are well
defined.  The change of variables \(s=ct\) gives
\begin{equation}\label{eq:Omega-phi-scaling}
 \Omega(c\alpha,c\beta,c)=c\phi(\alpha,\beta).
\end{equation}

This continuous notation is exactly the prime saving of
Section~\ref{sec:prime-saving}.  Indeed, for an integral triple
\((a,b,c)\), let \(E(a,b,c)\) be the subset of one period selected by
\eqref{eq:local-chi-definition}, and write it as a disjoint union of
half-open intervals \([\ell,r)\).  Splitting the integral into periods
and integrating \(t^{-2}\) gives
\begin{equation}\label{eq:Omega-equals-omega}
 \begin{split}
 \Omega(a,b,c)
 &=\sum_{k=0}^{\infty}\int_{E(a,b,c)}
       \frac{du}{(k+u)^2}\\
 &=\sum_{[\ell,r)\subset E(a,b,c)}\sum_{k=0}^{\infty}
   \left(\frac1{k+\ell}-\frac1{k+r}\right)
 =\omega(a,b,c).
 \end{split}
\end{equation}
Thus the normalized arithmetic saving is precisely
\(\omega/c=\phi\), not merely a formally similar integral.

For \(p=(\alpha,\beta)\), set
\begin{equation}\label{eq:local-general-phase}
 F_p(y)=\alpha\log|y|+\beta\log|y^2+6y+25|
       -\log|25-y|.
\end{equation}
Clearing the denominators in its stationary equation gives
\begin{equation}\label{eq:parameter-stationary-polynomial}
 \begin{split}
 P(y;p)={}&(\alpha+2\beta-1)y^3
 -(19\alpha+44\beta+6)y^2\\
 &-(125\alpha+150\beta+25)y-625\alpha.
 \end{split}
\end{equation}
In the neighborhood used below, \(P\) has a real root \(y_+(p)>25\)
and an upper-half-plane root \(y_c(p)\).  Define
\begin{equation}\label{eq:parameter-r-s}
 r(p)=F_p(y_+(p)),\qquad s(p)=F_p(y_c(p)).
\end{equation}
The smooth part of the clearing-denominator rate and the full rate are
\begin{align}
 G(p)&=2\alpha+4\beta-2
       -\left(5\beta-\frac52\right)\log2,
 \label{eq:local-G-definition}\\
 \mathcal C(p)&=G(p)-\phi(p).
 \label{eq:local-C-definition}
\end{align}
Indeed, \(2\alpha+4\beta-2=d_0/c\), whereas
\(5\beta-5/2=h/c\).  Consequently the parameter-dependent version of
\eqref{eq:pi-measure-formula} is
\begin{equation}\label{eq:local-B-definition}
 \mathcal B(p)=1+
 \frac{r(p)+\mathcal C(p)}{-s(p)-\mathcal C(p)},
\end{equation}
where the denominator is positive.  We shall prove that this condition,
the saddle construction, and all four strict inequalities in
\eqref{eq:admissible-chamber} hold on one common neighborhood of \(p_*\).

\subsection{Finite endpoint geometry}

Fix positive integers \(a,b,c\), put
\begin{equation}\label{eq:q-fourth-endpoint}
 q=a+2b-c\ne0,
\end{equation}
and keep \(q\) distinct from the degree cutoff
\(d_0=2a+4b-2c=2q\).  For a translation \(z=(\xi,\eta)\), define
\begin{equation}\label{eq:translated-cell-measure}
 \mathfrak m(\xi,\eta)=\int_0^1
 \chi(au+\xi,bu+\eta,cu)\,du.
\end{equation}

\begin{lemma}[Uniform local polyhedrality]\label{lem:uniform-polyhedrality}
There are \(\rho>0\), \(L<\infty\), and a continuous, positively
homogeneous, piecewise-linear function \(D:\mathbb R^2\to\mathbb R\)
such that, for every \(v\) with \(\|v\|_1=1\),
\begin{equation}\label{eq:m-exact-directional}
 \mathfrak m(sv)-\mathfrak m(0)=sD(v)
 \qquad(0\leq s\leq\rho).
\end{equation}
The same \(\rho\) works for all directions.  If
\(G_z(u)=\chi(au+z_1,bu+z_2,cu)\), then, after decreasing \(\rho\) if
necessary,
\begin{equation}\label{eq:Gz-Lipschitz}
 \int_0^1|G_z(u)-G_{z'}(u)|\,du
 \leq L\|z-z'\|_1
\end{equation}
whenever \(z,z'\) lie in the \(\rho\)-ball.  In the same ball one also
has the affine-translation estimate
\begin{equation}\label{eq:Gz-affine-Lipschitz}
 \int_0^1|G_{z+uw}(u)-G_z(u)|\,du
 \leq L\|w\|_1.
\end{equation}
\end{lemma}

\begin{proof}
The integrand can change only when a fractional part jumps or when the
strict inequality in \eqref{eq:local-chi-definition} is an equality.
After absorbing the relevant floor values in an integer \(j\), all such
endpoints belong to the four affine families
\begin{equation}\label{eq:four-endpoint-families}
 A_j(z)=\frac{j-1/2-\xi}{a},\quad
 B_j(z)=\frac{j-\eta}{b},\quad
 C_j=\frac jc,\quad
 Q_j(z)=\frac{j-1/2-\xi-2\eta}{q}.
\end{equation}
Only finitely many members meet a fixed neighborhood of \([0,1]\).
It is convenient first to regard \(u\) on the circle
\(\mathbb R/\mathbb Z\).  Choose a cut outside the finite base endpoint
set.  Distinct endpoints at \(z=0\) remain distinct in one fixed ball.
Endpoints tied at \(z=0\) are ordered by the finitely many homogeneous
linear differences of their velocities.  These equality lines form a
finite central fan, and the circular endpoint order is fixed on every
cone of that fan.  Persistent ties may be ordered by the fixed priority
\(A<B<C<Q\), followed by the integer indices; this only orders
zero-length arcs and does not alter their measure.

Between consecutive endpoints the indicator is constant.  Hence on
each cone \(\mathfrak m(z)\) is a sum of lengths of half-open arcs with
affine endpoints and therefore has the form
\(\mathfrak m(0)+\ell_C(z)\), where \(\ell_C\) is linear.  All cones
share the same value at their vertex.  Moving an endpoint by \(h\)
changes the selected set by \(O(|h|)\), so the measure is continuous;
the adjacent linear forms consequently agree on every common wall.
This defines the continuous homogeneous function \(D\) and proves
\eqref{eq:m-exact-directional} on one common ball.

For \eqref{eq:Gz-affine-Lipschitz}, the first, second, and fourth
families acquire the slopes
\[
 a+w_1,\qquad b+w_2,\qquad q+w_1+2w_2
\]
in \(u\), while the third keeps slope \(c\).  In a fixed small ball all
these slopes are bounded away from zero.  Solving the endpoint equations
shows that every corresponding endpoint moves by \(O(\|w\|_1)\),
uniformly in \(z\).  The indicators agree outside the union of the
resulting finitely many short intervals.  If a moving endpoint enters or
leaves \([0,1]\), it lies within \(O(\|w\|_1)\) of \(0\) or \(1\) and
contributes the same order of error; equivalently, one may keep the
circle description and include the cut as a fixed endpoint.  This proves
\eqref{eq:Gz-affine-Lipschitz}.  The same argument for two constant
translations proves \eqref{eq:Gz-Lipschitz}.
\end{proof}

The proof also gives an exact finite formula for \(D\).  Fix a direction
inside one cone and write the sorted germs as
\(e_j(s)=e_j(0)+s\lambda_j(v)\), cyclically after the chosen cut.  If
\(\epsilon_j\in\{0,1\}\) is the indicator on
\([e_j(s),e_{j+1}(s))\), then
\begin{equation}\label{eq:D-endpoint-formula}
 \boxed{D(v)=\sum_j\epsilon_j
       \{\lambda_{j+1}(v)-\lambda_j(v)\}.}
\end{equation}
Every number in this formula is rational for an integral triple; no
finite-difference approximation is involved.

At the candidate triple there are \(14856\) endpoint germs and \(12992\)
distinct circular base positions.  Direct comparison of the six pairs
of denominators in \eqref{eq:four-endpoint-families} gives the minimum
nonzero circular gap
\begin{equation}\label{eq:local-minimum-endpoint-gap}
 \Delta=\frac1{13797510}=\frac1{3714\cdot3715}.
\end{equation}
For the \(\ell^1\)-norm, the largest dual norm of an endpoint gradient is
\[
 M_0=\max\left\{\frac1{1857},\frac1{3714},
                       \frac2{3715}\right\}=\frac1{1857}.
\]
Thus \(\rho=1/60000\), for example, is a common separation radius:
\(2M_0\rho<\Delta/4\).  Coincident germs are handled by the central
fan and do not constrain this estimate.

\subsection{Uniform logarithmic variation of the prime saving}

\begin{lemma}[Uniform logarithmic variation]\label{lem:uniform-log-variation}
Let \(p_0=(a/c,b/c)\).  Uniformly for \(\|v\|_1=1\), as
\(\varepsilon\downarrow0\),
\begin{equation}\label{eq:uniform-log-variation}
 \phi(p_0+\varepsilon v)-\phi(p_0)
 =D(v)\varepsilon\log(1/\varepsilon)+O(\varepsilon).
\end{equation}
Equivalently, uniformly as \(h\to0\), \(h\ne0\),
\begin{equation}\label{eq:uniform-log-variation-h}
 \phi(p_0+h)-\phi(p_0)
 =D(h)\log(1/\|h\|_1)+O(\|h\|_1).
\end{equation}
\end{lemma}

\begin{proof}
We first perturb \(\Omega\).  Put \(0<\delta<\rho/2\), where \(\rho\)
is the common radius in Lemma~\ref{lem:uniform-polyhedrality}, and let
\(\|v\|_1=1\).  On the period \([k,k+1]\), integrality of \(a,b,c\) and
the substitution \(t=k+u\) make the contribution to the difference
\(\Omega(a+\delta v_1,b+\delta v_2,c)-\Omega(a,b,c)\) equal to
\begin{equation}\label{eq:local-period-block}
 I_k=\int_0^1
 \{G_{\delta(k+u)v}(u)-G_0(u)\}\frac{du}{(k+u)^2}.
\end{equation}
For
\[
 1\leq k\leq K:=\left\lfloor\frac{\rho}{2\delta}\right\rfloor,
\]
both translations needed below lie in the common ball.  Equations
\eqref{eq:Gz-affine-Lipschitz} and \eqref{eq:Gz-Lipschitz} give,
uniformly in \(v,k,\delta\),
\begin{align}
 \int_0^1|G_{\delta(k+u)v}-G_{\delta kv}|\,du&=O(\delta),
 \label{eq:local-block-affine-error}\\
 \int_0^1|G_{\delta kv}-G_0|\,du&=O(\delta k).
 \label{eq:local-block-translation-error}
\end{align}
Since
\[
 \sup_{0\leq u\leq1}|(k+u)^{-2}-k^{-2}|\leq2k^{-3},
\]
these estimates show that
\begin{equation}\label{eq:local-block-replacement}
 I_k=\frac{\mathfrak m(\delta kv)-\mathfrak m(0)}{k^2}
      +O(\delta k^{-2}).
\end{equation}
Exact polyhedrality now gives
\begin{equation}\label{eq:local-harmonic-sum}
 \begin{split}
 \sum_{k=1}^KI_k
 &=\delta D(v)\sum_{k=1}^K\frac1k+O(\delta)\\
 &=\delta D(v)\log(1/\delta)+O(\delta),
 \end{split}
\end{equation}
and every error constant is independent of the direction.  The remaining
tail is bounded by
\[
 \int_K^\infty t^{-2}\,dt=O(\delta).
\]

There is no singular error on \(0<t<1\).  On a sufficiently short fixed
initial interval no fractional part wraps, and the defining inequality
would read
\[
 \frac12+\bigl(q+\delta(v_1+2v_2)\bigr)t<0,
\]
which is uniformly impossible.  On the rest of the first period the
weight is bounded, and the finite-endpoint estimate gives \(O(\delta)\).
Thus
\[
 \Omega(a+\delta v_1,b+\delta v_2,c)-\Omega(a,b,c)
 =\delta D(v)\log(1/\delta)+O(\delta)
\]
uniformly on the unit diamond.  Finally take \(\delta=c\varepsilon\),
divide by \(c\), and use \eqref{eq:Omega-phi-scaling}.  The additional
term \(-\varepsilon D(v)\log c\) is uniformly \(O(\varepsilon)\), since
\(D\) is bounded on the unit diamond.  This proves
\eqref{eq:uniform-log-variation}; positive homogeneity gives
\eqref{eq:uniform-log-variation-h}.
\end{proof}

\subsection{Why the old exponent improves}

At the Zeilberger--Zudilin point take \((a,b,c)=(1,2,3)\), so that
\[
 p_{\rm ZZ}=\left(\frac13,\frac23\right),\qquad
 (A_1,A_2)=\left(\frac23,\frac23\right).
\]
A direct partition at \(1/3,1/2,2/3\) shows that, up to endpoints,
\begin{equation}\label{eq:old-accepted-set}
 E_0=[1/2,2/3),\qquad \mathfrak m(0,0)=\frac16.
\end{equation}
The directional change can also be read without any limiting argument.
For \(0<e<1/15\), set \((\xi,\eta)=e(1/2,1)\).  On the five open cells
cut out by
\[
 0<\frac13<\frac12-\frac e2<\frac23<1-\frac e2<1,
\]
the signed difference
\[
 \{u+1/2+e/2\}+2\{2u+e\}-\{3u\}
\]
is, respectively,
\[
 2u+\frac12+\frac52e,\quad
 2u+\frac32+\frac52e,\quad
 2u-\frac32+\frac52e,\quad
 2u-\frac12+\frac52e,\quad
 2u-\frac52+\frac52e.
\]
Their signs are \(+,+,-,+,-\).  Hence the complete accepted set is
\begin{equation}\label{eq:old-translated-accepted-set}
 E_e=[1/2-e/2,2/3)\mathbin{\dot\cup}[1-e/2,1),
\end{equation}
and therefore
\begin{equation}\label{eq:old-directional-D}
 \mathfrak m(e/2,e)=\frac16+e,\qquad D(1/2,1)=1.
\end{equation}
The second interval in \eqref{eq:old-translated-accepted-set}, born at
the period endpoint, accounts for half of this derivative.

Lemma~\ref{lem:uniform-log-variation} now yields
\begin{equation}\label{eq:old-phi-improvement}
 \phi\left(\frac13+\frac\varepsilon2,
            \frac23+\varepsilon\right)-
 \phi\left(\frac13,\frac23\right)
 =\varepsilon\log(1/\varepsilon)+O(\varepsilon).
\end{equation}
The real and complex roots of \(P(y;p_{\rm ZZ})\) are simple, because
the discriminant of \(2y^3-125y^2-500y-625\) is
\(-1425000000\).  Thus \(r,s,G\) have ordinary \(O(\varepsilon)\)
variation along this one-sided perturbation.  If
\[
 \tau_{\rm ZZ}=-s_{\rm ZZ}-\mathcal C_{\rm ZZ},\qquad
 K_{\rm ZZ}=\frac{r_{\rm ZZ}-s_{\rm ZZ}}{\tau_{\rm ZZ}^2},
\]
the finite logarithm calculation in
Appendix~\ref{subsec:old-point-sign} proves
\begin{equation}\label{eq:K-ZZ-rigorous-sign}
 11<K_{\rm ZZ}<12.
\end{equation}
In particular, the derivative of \(\mathcal B\) with respect to the
saving \(\phi\), with the smooth variables fixed, is
\begin{equation}\label{eq:B-phi-derivative}
 \frac{\partial\mathcal B}{\partial\phi}
 =-\frac{r-s}{(-s-\mathcal C)^2}<0
\end{equation}
at the old point.  Combining \eqref{eq:old-phi-improvement} and
\eqref{eq:B-phi-derivative}, we obtain the rigorous one-sided expansion
\begin{equation}\label{eq:old-B-improvement-expansion}
 \mathcal B\left(p_{\rm ZZ}+\varepsilon(1/2,1)\right)
 =\mathcal B(p_{\rm ZZ})
 -K_{\rm ZZ}\varepsilon\log(1/\varepsilon)+O(\varepsilon).
\end{equation}
Since the logarithm tends to infinity, this proves
\eqref{eq:old-point-variation}.  Notice that only
\(\varepsilon>0\) is used: the perturbation enters the open half-space
\(\alpha+\beta>1\) from its boundary.

\subsection{The tight-slack diagonal refinement}

The preceding calculation identifies a direction, not a finite rational
point.  On the original diagonal \(A_1=A_2\), equivalently \(b=2a\),
write \(a=M\), \(b=2M\).  The nearest integral choice on the admissible
side of the old equality \(a+b=c\) is
\begin{equation}\label{eq:tight-slack-family}
 a=M,\qquad b=2M,\qquad c=3M-1.
\end{equation}
Its relevant integer slacks are
\begin{align*}
 a+b-c&=1,&2b-c&=M+1,\\
 (5c-1)-7b&=M-6,
 &(2a+4b-2c)-(c+1)&=M+2.
\end{align*}
Thus \(M\geq6\) satisfies all the arithmetic inequalities.  Moreover,
\begin{equation}\label{eq:tight-slack-exponent}
 A_1=A_2=\frac{2M}{3M-1}
 =\frac23+\frac{2}{3(3M-1)}.
\end{equation}
Here the fourth endpoint parameter and the degree cutoff are, separately,
\begin{equation}\label{eq:tight-slack-q-d0}
 q=a+2b-c=2M+1,\qquad d_0=2q=4M+2.
\end{equation}
The stationary polynomial in this family is
\begin{equation}\label{eq:tight-slack-stationary}
 (2M+1)y^3-(125M-6)y^2-(500M-25)y-625M.
\end{equation}

The following few values record how the final scale was found; they are
not used to prove the theorem.
\begin{center}
\begin{tabular}{ccl}
\toprule
\(M\)&common exponent \(A_1=A_2\)&exploratory value of \(\mathcal B\)\\
\midrule
\(\infty\)&\(2/3\)&\(7.1032053341370017\ldots\)\\
\(100\)&\(200/299\)&\(7.151305465560188\ldots\)\\
\(1000\)&\(2000/2999\)&\(7.102252364091047\ldots\)\\
\(1857\)&\(1857/2785\)&\(7.101862832356351\ldots\)\\
\bottomrule
\end{tabular}
\end{center}
A refined exploration of \eqref{eq:tight-slack-family} suggested
\(M=1857\).  It gives
\begin{equation}\label{eq:new-exponent-from-family}
 A_1=A_2=\frac{3714}{5570}=\frac{1857}{2785}
 =\frac23+\frac1{8355},
\end{equation}
and, in normalized coordinates,
\begin{equation}\label{eq:new-point-old-direction}
 p_*=p_{\rm ZZ}+\frac1{8355}(1/2,1).
\end{equation}
This table is only a discovery record: it neither exhausts all \(M\)
nor searches all off-diagonal parameters.  The fixed-parameter proof in
Sections~\ref{sec:arithmetic}--\ref{sec:measure} establishes the bound at
the selected rational point.  The local theorem below is likewise not a
global-optimization claim.

\subsection{The twelve exact directional sectors}

We now specialize the endpoint calculation to
\((a,b,c)=(1857,3714,5570)\); then \(q=3715\).  The four velocities in a
direction \(v=(v_1,v_2)\) are
\begin{equation}\label{eq:four-endpoint-velocities}
 \lambda_A=-\frac{v_1}{1857},\qquad
 \lambda_B=-\frac{v_2}{3714},\qquad
 \lambda_C=0,\qquad
 \lambda_Q=-\frac{v_1+2v_2}{3715}.
\end{equation}
Their six pairwise equality lines and primitive rays are
\begin{equation}\label{eq:six-fan-walls}
 \begin{array}{c|c}
 \text{line}&\text{rays}\\ \hline
 v_2=2v_1&\pm(1,2)\\
 v_1=0&\pm(0,1)\\
 929v_1=1857v_2&\pm(1857,929)\\
 v_2=0&\pm(1,0)\\
 3714v_1+3713v_2=0&\pm(3713,-3714)\\
 v_1+2v_2=0&\pm(2,-1).
 \end{array}
\end{equation}
In the counterclockwise order used below, consecutive determinants are
\[
 929,2785,1,3713,3715,1,929,2785,1,3713,3715,1,
\]
so these are exactly twelve nondegenerate sectors.

For completeness we derive the sector coefficients from the endpoint
data.  The complete catalogue of coincident base positions is
\begin{equation}\label{eq:tie-catalogue}
 \begin{array}{c|l}
 \text{types}&\text{common positions on }\mathbb R/\mathbb Z\\ \hline
 A,B&\text{the \(1857\) odd points of the \(B\)-grid}\\
 A,C&1/2\\
 A,Q&1/2\\
 B,C&0,1/2\\
 B,Q&1/2\\
 C,Q&1/10,3/10,1/2,7/10,9/10.
 \end{array}
\end{equation}
These statements follow by cross-multiplying the corresponding members
of \eqref{eq:four-endpoint-families}; the coprimalities
\(\gcd(1857,3715)=\gcd(3714,3715)=1\) leave precisely the displayed
solutions.

For a sorted cyclic endpoint list, let \(\epsilon_j\) be the indicator
on \([e_j,e_{j+1})\), and define the signed endpoint counts
\begin{equation}\label{eq:signed-endpoint-counts}
 n_T=\sum_{e_j\ {\rm of\ type}\ T}(\epsilon_{j-1}-\epsilon_j),
 \qquad T\in\{A,B,C,Q\}.
\end{equation}
Then \eqref{eq:D-endpoint-formula} becomes
\begin{equation}\label{eq:D-from-signed-counts}
 D(v)=n_A\lambda_A(v)+n_B\lambda_B(v)
      +n_C\lambda_C(v)+n_Q\lambda_Q(v).
\end{equation}
The five interval families in Proposition~\ref{prop:E-decomposition}
determine every nonzero base arc.  At a tie, the order of the four
velocities in \eqref{eq:four-endpoint-velocities} determines whether a
zero-length base arc opens under the perturbation.  Counting the left and
right ends of those intervals gives
\[
 \mathfrak m(0)=|E|=\frac{1724689}{13797510}
\]
and the following exact table; the last two columns are obtained from
\eqref{eq:D-from-signed-counts}.

\begin{center}
\small
\begin{tabular}{c|rrrr|cc}
\toprule
\(j\)&\(n_A\)&\(n_B\)&\(n_C\)&\(n_Q\)&\(p_x\)&\(p_y\)\\
\midrule
1&0&-2786&1671&1115&\(-223/743\)&\(206777/1379751\)\\
2&0&-2786&1671&1115&\(-223/743\)&\(206777/1379751\)\\
3&-1857&-929&1671&1115&\(520/743\)&\(-966197/2759502\)\\
4&-1856&-929&1670&1115&\(964897/1379751\)&\(-966197/2759502\)\\
5&-1856&-930&1670&1116&\(4822628/6898755\)&\(-805783/2299585\)\\
6&-1856&-930&1673&1113&\(4828199/6898755\)&\(-802069/2299585\)\\
7&-1856&-928&1671&1113&\(4828199/6898755\)&\(-2409922/6898755\)\\
8&-1857&-928&1671&1114&\(2601/3715\)&\(-2413636/6898755\)\\
9&0&-2785&1671&1114&\(-1114/3715\)&\(2071483/13797510\)\\
10&0&-2785&1671&1114&\(-1114/3715\)&\(2071483/13797510\)\\
11&0&-2784&1671&1113&\(-1113/3715\)&\(345866/2299585\)\\
12&0&-2784&1669&1115&\(-223/743\)&\(68678/459917\)\\
\bottomrule
\end{tabular}
\end{center}
Here \(D(v)=p_xv_1+p_yv_2\) in sector \(j\), and every row has the
cyclic checksum \(n_A+n_B+n_C+n_Q=0\).  To illustrate the count rather
than merely quote it, consider sector 1 and its interior vector
\(v=(2,1)\).  Its velocity order is
\(\lambda_A<\lambda_Q<\lambda_B<\lambda_C\).  Applying this order at
every position in \eqref{eq:tie-catalogue} and tallying the five explicit
index ranges of Proposition~\ref{prop:E-decomposition} gives \(2786\)
\(B\)-type left ends, \(186+557+928=1671\) \(C\)-type right ends, and
\(742+371+2=1115\) \(Q\)-type right ends.  There are no other endpoint
types in the tally.  Hence
\((n_A,n_B,n_C,n_Q)=(0,-2786,1671,1115)\), and
\begin{equation}\label{eq:sample-sector-calculation}
 \begin{split}
 D(v)&=-2786\lambda_B(v)+1115\lambda_Q(v)\\
 &=-\frac{223}{743}v_1
   +\frac{206777}{1379751}v_2.
 \end{split}
\end{equation}
The remaining rows follow in exactly the same way from
\eqref{eq:tie-catalogue}; all operations are integer counts and rational
cross-multiplications.

For clarity, here are all twelve sectors and their linear forms.  The
rays are in counterclockwise order, and each row includes both boundary
rays by continuity.
\begin{center}
\small
\begin{tabular}{c|c|c|l}
\toprule
\(j\)&initial ray&terminal ray&\(D(v)\)\\
\midrule
1&\((1,0)\)&\((1857,929)\)&\(-223v_1/743+206777v_2/1379751\)\\
2&\((1857,929)\)&\((1,2)\)&\(-223v_1/743+206777v_2/1379751\)\\
3&\((1,2)\)&\((0,1)\)&\(520v_1/743-966197v_2/2759502\)\\
4&\((0,1)\)&\((-3713,3714)\)&\(964897v_1/1379751-966197v_2/2759502\)\\
5&\((-3713,3714)\)&\((-2,1)\)&\(4822628v_1/6898755-805783v_2/2299585\)\\
6&\((-2,1)\)&\((-1,0)\)&\(4828199v_1/6898755-802069v_2/2299585\)\\
7&\((-1,0)\)&\((-1857,-929)\)&\(4828199v_1/6898755-2409922v_2/6898755\)\\
8&\((-1857,-929)\)&\((-1,-2)\)&\(2601v_1/3715-2413636v_2/6898755\)\\
9&\((-1,-2)\)&\((0,-1)\)&\(-1114v_1/3715+2071483v_2/13797510\)\\
10&\((0,-1)\)&\((3713,-3714)\)&\(-1114v_1/3715+2071483v_2/13797510\)\\
11&\((3713,-3714)\)&\((2,-1)\)&\(-1113v_1/3715+345866v_2/2299585\)\\
12&\((2,-1)\)&\((1,0)\)&\(-223v_1/743+68678v_2/459917\)\\
\bottomrule
\end{tabular}
\end{center}
Substitution at the rays gives, without rounding,
\begin{center}
\begin{tabular}{c|r@{\qquad}c|r}
\toprule
ray&\(D(r)\)&ray&\(D(r)\)\\
\midrule
\((1,0)\)&\(-223/743\)&\((-1,0)\)&\(-4828199/6898755\)\\
\((1857,929)\)&\(-776458/1857\)&\((-1857,-929)\)&\(-1810807/1857\)\\
\((1,2)\)&\(-557/1379751\)&\((-1,-2)\)&\(-557/1379751\)\\
\((0,1)\)&\(-966197/2759502\)&\((0,-1)\)&\(-2071483/13797510\)\\
\((-3713,3714)\)&\(-7236730/1857\)&\((3713,-3714)\)&\(-1671\)\\
\((-2,1)\)&\(-3247/1857\)&\((2,-1)\)&\(-464/619\)\\
\bottomrule
\end{tabular}
\end{center}
The two adjacent forms agree on each ray.  Exact comparison of these
twelve fractions gives
\begin{equation}\label{eq:exact-kappa}
 \min_r\frac{-D(r)}{\|r\|_1}
 =\kappa:=\frac{557}{4139253},
\end{equation}
with equality at \(r=(1,2)\) and \(r=(-1,-2)\).  If \(v\) lies in the
cone spanned by consecutive rays \(r,r'\), write
\(v=\lambda r+\mu r'\) with \(\lambda,\mu\geq0\).  Linearity on the
cone and the triangle inequality imply
\begin{equation}\label{eq:uniform-D-negative}
 \begin{split}
 D(v)&=\lambda D(r)+\mu D(r')\\
 &\leq-\kappa\{\lambda\|r\|_1+\mu\|r'\|_1\}\\
 &\leq-\frac{557}{4139253}\|v\|_1.
 \end{split}
\end{equation}
In particular \(D(v)<0\) for every nonzero real direction \(v\).

\subsection{Persistence of the analytic construction}

We next verify that the analytic quantities just used really belong to
one neighborhood, rather than to unrelated pointwise constructions.
At \(p_*\), multiplying \eqref{eq:parameter-stationary-polynomial} by
\(5570\) gives the cubic \eqref{eq:stationary-cubic}.  Its negative
discriminant and the exact signs
\[
 5570P(66;p_*)=-5502699<0,\qquad
 5570P(67;p_*)=11983904>0
\]
show that its real root lies in \((66,67)\), while the other two roots
form a simple conjugate pair.  The real-analytic implicit-function
theorem supplies branches \(y_+(p)\) and \(y_c(p)\) on one neighborhood
of \(p_*\).  Shrinking it preserves
\[
 y_+(p)>25,\qquad \Im y_c(p)>0,
\]
and separation from \(0,25,-3\pm4i\).  Hence \(r(p)\) and \(s(p)\) in
\eqref{eq:parameter-r-s} are real analytic there.  The positive root is
equivalent to a unique coefficient saddle \(\zeta(p)\in(0,1)\) through
\[
 y_+(p)=25\frac{(1+\zeta(p))^2}{(1-\zeta(p))^2};
\]
the positive-coefficient and positive-variance argument of
Lemma~\ref{lem:positive-coefficient-powers} consequently persists for
nearby rational parameters.

We also retain the full contour construction, including its endpoints.
To avoid confusing it with the integer \(q\) in
\eqref{eq:q-fourth-endpoint}, write \(z_0=-3+4i\), and put
\begin{equation}\label{eq:local-persistent-arc}
 Z(p)=\frac{z_0}{y_c(p)},\quad
 \lambda_*(p)=\frac1{\Re Z(p)},\quad
 \eta_*(p)=-\frac{\Im Z(p)}{\Re Z(p)-1},\quad
 \gamma_p(\lambda)=
 \frac{z_0\lambda}{1-i\eta_*(p)(1-\lambda)}.
\end{equation}
The strict isolations \eqref{eq:eta-short-isolation} and
\eqref{eq:lambda-short-isolation} persist by continuity, say with
\[
 \frac9{10}<\eta_*(p)<\frac{23}{25}<\frac43,\qquad
 \frac{12}{25}<\lambda_*(p)<\frac{49}{100}.
\]
The definitions give \(\gamma_p(\lambda_*(p))=y_c(p)\).  Moreover,
\begin{equation}\label{eq:local-persistent-arc-geometry}
 \Im\gamma_p(\lambda)=
 \frac{\lambda\{4-3\eta_*(p)(1-\lambda)\}}
 {1+\eta_*(p)^2(1-\lambda)^2}>0,\qquad
 |\gamma_p(\lambda)|\leq5.
\end{equation}

Take the same oriented square-root lifts as in
\eqref{eq:square-root-lifts}, now depending on \(p\):
\[
 \tau_{+,p}(\lambda)=-\sqrt{\gamma_p(\lambda)},\qquad
 \tau_{-,p}(\lambda)=\overline{\tau_{+,p}(\lambda)}.
\]
They join \(-1-2i\) to \(0\) and then \(0\) to \(-1+2i\), and
\(|\tau_{\pm,p}|\leq\sqrt5<5\).  The original vertical segment and the
lifts therefore remain in the pole-free disk \(|t|<5\); Cauchy's theorem
gives the same homotopy as in Section~\ref{sec:asymptotics}.  From
\eqref{eq:gamma-derivative-origin}, uniformly after shrinking the
parameter neighborhood,
\[
 |\tau_{+,p}'(\lambda)|=O(\lambda^{-1/2})
 \qquad(\lambda\downarrow0),
\]
so the lifts are rectifiable with locally uniform length bounds.

Let \(Q_p(\lambda)\) denote the degree-six numerator obtained by clearing
the strictly positive factors from
\(d\{F_p(\gamma_p(\lambda))\}/d\lambda\).  At \(p_*\), the coefficients
of the transformed polynomial
\begin{equation}\label{eq:local-descartes-transform}
 (1+u)^6Q_{p_*}\left(\frac{u}{1+u}\right)
\end{equation}
have the strict signs \((+,+,+,+,-,-,-)\), as proved in
\eqref{eq:Descartes-coefficients}.  These are monomial coefficients of
the transformed polynomial, not Bernstein coefficients.  They depend
real analytically on \(p\), and their signs therefore persist.  Descartes'
rule permits at most one critical point on \(0<\lambda<1\); the saddle
identity supplies one, and the endpoint phase tends to \(-\infty\).
Thus \(y_c(p)\) remains the unique global phase maximum on the arc.

Finally, for a nearby rational parameter represented by integers
\((a,b,c)\), the exact analogue of \eqref{eq:fixed-amplitude-identity} is
\begin{equation}\label{eq:local-uniform-amplitude}
 |R_n(\tau_{+,p}(\lambda))|
 =\frac5{|25-\gamma_p(\lambda)|}
   \exp\{cnF_p(\gamma_p(\lambda))\}.
\end{equation}
Equation \eqref{eq:local-persistent-arc-geometry} gives
\(|25-\gamma_p|\geq20\).  The rectifiability estimate is locally
uniform, so Lemma~\ref{lem:elementary-contour-estimate} gives a constant
\(C_0\), independent of \(n\) and locally uniform in \(p\), such that
\begin{equation}\label{eq:local-uniform-integral-bound}
 |J_n|\leq C_0e^{cns(p)}.
\end{equation}
This proves that the same analytic branch continues to be the contour
rate; it also records explicitly the square-root lift, homotopy,
rectifiability, and fixed-amplitude ingredients used in that conclusion.

It remains only to keep the construction inside the arithmetic chamber.
At \(p_*\), the four margins are
\begin{equation}\label{eq:local-chamber-margins}
 \alpha_*+\beta_*-1=\frac1{5570},\quad
 2\beta_*-1=\frac{929}{2785},\quad
 5-7\beta_*=\frac{926}{2785},\quad
 2\alpha_*+4\beta_*-3=\frac{186}{557}.
\end{equation}
Taking the minimum of each margin divided by the \(\ell^\infty\)-norm of
its linear normal gives a positive \(\ell^1\)-radius contained in all
four half-spaces.  The enclosures of Sections~\ref{sec:prime-saving}
and \ref{sec:asymptotics} also give the strict bounds
\begin{equation}\label{eq:local-positive-rate-margins}
 r(p_*)+\mathcal C(p_*)>3,\qquad
 -s(p_*)-\mathcal C(p_*)>\frac35.
\end{equation}
Lemma~\ref{lem:uniform-log-variation} makes \(\phi\), hence
\(\mathcal C\), continuous at \(p_*\); the other terms are analytic.
After one final shrinking, the saddle, contour, chamber, and positive
denominator conditions therefore hold simultaneously.

\subsection{Strict local minimality}

\begin{theorem}\label{thm:local-minimum-detailed}
There is \(\delta>0\) such that every real \(p\ne p_*\) with
\(\|p-p_*\|_1<\delta\) lies in the admissible chamber and satisfies
\[
 \mathcal B(p)>\mathcal B(p_*).
\]
Thus \(p_*\) is a strict two-dimensional local minimizer, with no
restriction to rational points, straight rays, or a finite list of
directions.
\end{theorem}

\begin{proof}
Separate the only nonsmooth term by defining
\begin{equation}\label{eq:local-H-definition}
 H(p,w)=1+\frac{r(p)+G(p)-w}{-s(p)-G(p)+w}.
\end{equation}
Then \(\mathcal B(p)=H(p,\phi(p))\).  At the candidate,
\begin{equation}\label{eq:local-K-star}
 \partial_wH(p_*,\phi(p_*))
 =-\frac{r(p_*)-s(p_*)}
         {\{-s(p_*)-\mathcal C(p_*)\}^2}
 =:-K_*.
\end{equation}
Exact rational comparison using \eqref{eq:r-enclosure},
\eqref{eq:C-enclosure-recalled}, and \eqref{eq:s-enclosure-recalled}
gives
\begin{equation}\label{eq:local-K-star-bounds}
 11<K_*<12
 \qquad(K_*=11.1919651181036569\ldots).
\end{equation}

Let \(h=\varepsilon v\), where
\(\varepsilon=\|h\|_1\) and \(\|v\|_1=1\).  The remainder in
Lemma~\ref{lem:uniform-log-variation} is uniform in \(v\), and
\[
 |\phi(p_*+\varepsilon v)-\phi(p_*)|
 =O(\varepsilon\log(1/\varepsilon))
\]
uniformly on the unit diamond.  All first and second derivatives of \(H\)
are bounded on one compact subneighborhood.  Taylor's theorem in
\((p,w)\), with \eqref{eq:uniform-log-variation}, therefore gives
\begin{equation}\label{eq:uniform-B-Taylor}
 \begin{split}
 \mathcal B(p_*+\varepsilon v)-\mathcal B(p_*)
 &=-K_*D(v)\varepsilon\log(1/\varepsilon)+O(\varepsilon)\\
 &\quad+O(\varepsilon^2\log^2(1/\varepsilon)),
 \end{split}
\end{equation}
where every constant is independent of \(v\).  The quadratic remainder
is \(o(\varepsilon)\), and \eqref{eq:uniform-D-negative} supplies a
constant \(C<\infty\) such that
\begin{equation}\label{eq:uniform-B-lower-bound}
 \mathcal B(p_*+\varepsilon v)-\mathcal B(p_*)
 \geq K_*\frac{557}{4139253}
       \varepsilon\log(1/\varepsilon)-C\varepsilon.
\end{equation}
Choose \(\delta\) inside all the common neighborhoods constructed above
and so small that
\[
 K_*\frac{557}{4139253}\log(1/\delta)>C.
\]
Then \eqref{eq:uniform-B-lower-bound} is strictly positive whenever
\(0<\varepsilon<\delta\).  Because every displacement, including every
point on a nonlinear approach to \(p_*\), has the representation
\(h=\|h\|_1(h/\|h\|_1)\), this proves the assertion for all real nearby
points.  It proves Theorem~\ref{thm:local-minimum} as well.
\end{proof}

The original exponent coordinates are related by the invertible linear
map \((A_1,A_2)\mapsto(A_1/2,A_2)\), so the theorem is equivalent to
strict local minimality at
\[
 (A_{1,*},A_{2,*})=
 \left(\frac{1857}{2785},\frac{1857}{2785}\right).
\]
As a check on the diagonal direction, the sector table gives
\[
 D(1/2,1)=D(-1/2,-1)=-\frac{557}{2759502};
\]
the two one-sided leading coefficients are consequently
\begin{equation}\label{eq:diagonal-leading-coefficient}
 K_*\frac{557}{2759502}
 =0.0022590759386236\ldots>0.
\end{equation}

\begin{proposition}[Realization by the integral construction]
\label{prop:rational-realization}
Every rational point \(p=(\alpha,\beta)\) in a sufficiently small
neighborhood of \(p_*\) is represented by an admissible integer
construction whose normalized Hata bound is exactly \(\mathcal B(p)\).
\end{proposition}

\begin{proof}
At the candidate,
\(\mathfrak m(0)=|E|=1724689/13797510>0\) by
\eqref{eq:E-measure}.  Thus the positive weighted integral
\(\Omega(1857,3714,5570)\) is nonzero, and
\eqref{eq:Omega-phi-scaling} gives \(\phi(p_*)>0\).  By the continuity proved in
Lemma~\ref{lem:uniform-log-variation}, we may first shrink the
neighborhood so that \(\phi(p)>0\) throughout it.

Choose positive integers \(a,b\) and a sufficiently large even common
denominator \(c\) such that \(a=c\alpha\) and \(b=c\beta\).  Because the
four chamber inequalities are strict and all expressions are integral,
\begin{equation}\label{eq:general-integer-margins}
 a+b-c\geq1,\qquad 2b-c\geq1,\qquad
 7b\leq5c-1,\qquad d_0-c\geq1,
\end{equation}
where \(d_0=2a+4b-2c\).

We first check that the symbolic arithmetic proofs of
Section~\ref{sec:arithmetic} apply without alteration.  The inequalities
\(a+b>c\) and \(2b>c\) make the complete Laurent bounds integral: in
particular
\[
 4b-2c=2(2b-c)\geq2,\qquad
 2(a+b-c)\geq2.
\]
The inequality \(7b<5c\) is equivalent to
\[
 h:=5b-\frac{5c}{2}<\frac{3b}{2},
\]
which is precisely the estimate used for the endpoint coefficients with
\(k\geq bn\).  The condition \(d_0>c\) gives
\(cn\leq d_0n\), so the degree and least-common-multiple ranges in the
polynomial and pole integrations remain valid.  Finally, evenness of
\(c\) makes both \(cn\) even and \(h\) integral.  Thus the Laurent,
endpoint, reduced-LCM, polynomial, and pole arguments used in the proof
of Proposition~\ref{prop:arithmetic-normalization}, before its numerical
specialization, give
\begin{equation}\label{eq:general-arithmetic-multiplier}
 M_n(p):=2^{4-hn}
 \frac{\lcm(1,2,\ldots,d_0n)}{\Phi_n},
 \qquad
 M_n(p)J_n=U_n+V_n\pi,\quad U_n,V_n\in\mathbb Z,
 \quad V_n=M_n(p)v_n
\end{equation}
for this triple.

The upper cutoff in the removable-prime product also has the required
range.  Indeed,
\begin{align}
 d_0-2a&=4b-2c=2(2b-c)>0,\notag\\
 d_0-b&=2(a+b-c)+b>0,\label{eq:general-cutoff-margins}\\
 d_0-c&>0.\notag
\end{align}
Consequently, if \(0<u\leq1/d_0\), then
\[
 au<\frac12,\qquad bu<1,\qquad cu<1.
\]
None of \(au+1/2,bu,cu\) wraps modulo one, and the left side minus the
right side of the selecting inequality is
\begin{equation}\label{eq:general-initial-gap}
 \left(au+\frac12\right)+2bu-cu
 =\frac12+(a+2b-c)u>0.
\end{equation}
All four inequalities used here are strict at \(u=1/d_0\).  By
continuity they remain valid on
\((0,1/d_0+\epsilon_0]\) for some \(\epsilon_0>0\), and the same
no-wrap calculation still excludes every point of that interval.
It follows that
\[
 E(a,b,c)\cap(0,1/d_0+\epsilon_0]=\varnothing.
\]
On the other hand, \(\Omega(a,b,c)=c\phi(p)>0\), so the periodic set
\(E(a,b,c)\) is nonempty.  Its first endpoint therefore exists and
satisfies \(\ell_0\geq1/d_0+\epsilon_0>1/d_0\).  Lemma
\ref{lem:periodic-prime-product}, now with \(D=d_0\), therefore applies
to the actual product truncated at \(d_0n\) and proves
\begin{equation}\label{eq:general-prime-saving-bridge}
 \lim_{n\to\infty}\frac{\log\Phi_n}{cn}
 =\frac{\omega(a,b,c)}c=\phi(p).
\end{equation}
The normalized multiplier rate is consequently
\(\mathcal C(p)=G(p)-\phi(p)\).

It remains to identify both analytic rates.  The substitution
\eqref{eq:mobius-z-change} gives, for this general triple,
\begin{equation}\label{eq:general-positive-coefficient-identity}
 |v_n|=\frac14
 \left(5^{2(a+b-c)}2^{4b-2c}\right)^n
 [z^{cn}]
 \left\{\frac{(1+z)^{2a}P(z)^b}{(1-z)^{d_0}}\right\}^{\!n}.
\end{equation}
All coefficients of the function in braces are positive, its support has
span one, and its saddle variance is positive.  Lemma
\ref{lem:positive-coefficient-powers} therefore gives an ordinary
coefficient limit; in particular, the right side of
\eqref{eq:general-positive-coefficient-identity} is positive and
\(v_n\ne0\), hence \(V_n\ne0\).  Direct substitution of
\begin{equation}\label{eq:general-y-z-bridge}
 y=25\frac{(1+z)^2}{(1-z)^2}
\end{equation}
in its saddle equation yields
\eqref{eq:parameter-stationary-polynomial}; the positive saddle is
\(y_+(p)\), and simplification of the exponential value gives
\[
 \lim_{n\to\infty}\frac{\log|v_n|}{cn}
 =F_p(y_+(p))=r(p).
\]
On the complex side, the locally uniform square-root deformation and
amplitude estimate \eqref{eq:local-uniform-amplitude}--%
\eqref{eq:local-uniform-integral-bound} give
\[
 \limsup_{n\to\infty}\frac{\log|J_n|}{cn}\leq
 F_p(y_c(p))=s(p).
\]
Thus the general integer forms in
\eqref{eq:general-arithmetic-multiplier} have normalized rates
\[
 \sigma(p)=r(p)+\mathcal C(p),\qquad
 \tau(p)=-s(p)-\mathcal C(p)>0.
\]
Lemma~\ref{lem:hata-one-number} shows that their Hata bound is exactly
\[
 1+\frac{\sigma(p)}{\tau(p)}=\mathcal B(p),
\]
as claimed.
\end{proof}

Rational density is not used in the real local-minimum proof.
Proposition~\ref{prop:rational-realization} only identifies the arithmetic
meaning of nearby rational points.  Nothing here asserts global
optimality over this chamber, over stronger denominator savings, or over
other constructions.

\appendix
\section{Exact algebra and numerical enclosures}\label{sec:exact-calculations}

This appendix prints the finite algebra and the error bounds used in
Sections~\ref{sec:asymptotics} and \ref{sec:local-optimality}.  Every
displayed decimal endpoint below is a terminating decimal and hence an
exact rational number.

\subsection{The real saddle and the coefficient rate}
\label{subsec:real-saddle-calculations}

For \(m(z)=zS'(z)/S(z)\), exact simplification gives
\begin{equation}\label{eq:coefficient-saddle-numerator}
 (1-z^2)P(z)\{m(z)-5570\}=N(z),
\end{equation}
where
\[
 \begin{split}
 N(z)={}&-11140+11152z+102158z^2+167160z^3\\
        &\quad+102158z^4+11152z^5-11140z^6.
 \end{split}
\]
With \(w=(1+z)/(1-z)\), direct expansion gives
\begin{equation}\label{eq:N-to-stationary-cubic}
 (w+1)^6N\left(\frac{w-1}{w+1}\right)
 =\frac4{625}p(25w^2).
\end{equation}
Both sides, in ascending powers of \(w\), equal
\[
 -7428-148556w^2-928476w^4+371500w^6.
\]
The coefficients of \(p\) have one sign change.  Since \(p(0)<0\) and
\(p(y)\to+\infty\) as \(y\to+\infty\), Descartes' rule shows that
\(p\) has exactly one positive root.  Thus
\begin{equation}\label{eq:rho-zeta-relation}
 \rho=25\frac{(1+\zeta)^2}{(1-\zeta)^2}.
\end{equation}

For the contour calculations we use the rational endpoints
\[
 \rho_L=\frac{66321017380598567172217731240}{10^{27}},
 \qquad
 \rho_U=\frac{66321017380598567172217731241}{10^{27}}.
\]
Exact cross-multiplication gives the strict sign margins
\[
 -10^{-20}<p(\rho_L)<-9\cdot10^{-21}<0
 <8\cdot10^{-21}<p(\rho_U)<9\cdot10^{-21}.
\]
Moreover \(p'(66)=16979437>0\), and
\(p''(y)=22290y-464238>0\) on \([66,67]\).  Hence
\(\rho_L<\rho<\rho_U\).

For the finer coefficient evaluation, put
\[
 R_0=66.321017380598567172217731240525
\]
and define the exact rationals
\begin{align*}
 \widehat\rho_-&=R_0+
 31930114787749744963\cdot10^{-50},\\
 \widehat\rho_+&=R_0+
 31930114787749744964\cdot10^{-50}.
\end{align*}
Exact integer substitution gives
\begin{equation}\label{eq:rho-fine-isolation}
 p(\widehat\rho_-)<0<p(\widehat\rho_+),\qquad
 \widehat\rho_-<\rho<\widehat\rho_+.
\end{equation}
Using \eqref{eq:rho-zeta-relation}, exact rational squaring and
cross-multiplication then give
\[
 Z_0=0.239183376438453117718495639373,
\]
\begin{align}
 z_-&=Z_0+15852029007245250806\cdot10^{-50},\notag\\
 z_+&=Z_0+15852029007245250807\cdot10^{-50},\notag\\
 z_-&<\zeta<z_+.\label{eq:zeta-isolation}
\end{align}

Here is the elementary logarithm bound used both here and below.  For
\(|u|<1\), put
\[
 \Lambda_N(u)=2\sum_{j=0}^{N-1}\frac{u^{2j+1}}{2j+1}.
\]
The geometric bound for the omitted tail proves
\begin{equation}\label{eq:atanh-log-bound}
 \left|\log\frac{1+u}{1-u}-\Lambda_N(u)\right|
 \leq\frac{2|u|^{2N+1}}{(2N+1)(1-u^2)}.
\end{equation}
To bound the logarithm of a positive rational \(x\), choose the integer
\(k\) for which \(1\leq x/2^k<2\), set
\(u=(x-2^k)/(x+2^k)\), and use
\[
 \log x=k\log2+\log\frac{1+u}{1-u}.
\]
Then \(0\leq u<1/3\); the same formula with \(u=1/3\) bounds
\(\log2\).  At \(N=120\), the right side of
\eqref{eq:atanh-log-bound} is at most
\[
 \frac{2(1/3)^{241}}{241(1-1/9)}<9.64\cdot10^{-118}.
\]
Applying these finite rational sums to
\[
 2,\quad25,\quad1+z_\pm,\quad P(z_\pm),\quad
 1-z_\pm,\quad z_\pm
\]
and respecting the signs of the logarithmic terms in
\eqref{eq:r-definition} gives an interval of width less than
\(8.4\cdot10^{-50}\), contained in the interval
\eqref{eq:r-enclosure}.  This proves the stated value of \(r\) using
only the printed finite sums and their remainder bound.

\subsection{Construction and isolation of the complex arc}
\label{subsec:complex-arc-calculations}

Vieta's formulas reconstruct the upper nonreal root from the real one:
\begin{equation}\label{eq:Vieta-complex-root}
 U=\frac12\left(\frac{232119}{3715}-\rho\right),\qquad
 U^2+V^2=\frac{232125}{743\rho},\qquad
 V=\sqrt{\frac{232125}{743\rho}-U^2}.
\end{equation}
Rational interval arithmetic from \(\rho_L<\rho<\rho_U\), including
comparison of rational squares for \(V\), gives
\begin{equation}\label{eq:UV-isolation}
 \begin{aligned}
 -\frac{1919728071187574299433227}{10^{24}}
 &<U<
 -\frac{1919728071187574299433226}{10^{24}},\\
 \frac{1012573903136435847419597}{10^{24}}
 &<V<
 \frac{1012573903136435847419598}{10^{24}}.
 \end{aligned}
\end{equation}

We next give an exact algebraic description of \(\eta_*\).  In this
calculation \(\xi\) is an elimination variable, not the arc parameter
\(\lambda\).  Put
\[
 z=\xi+ie(1-\xi),\qquad y=\frac qz.
\]
At the common root corresponding to the desired saddle,
\(\xi=X=1/\lambda_*>1\) and
\(\Im z=Y=e(1-X)\).  The real and imaginary parts of the numerator of
\(p(q/z)\) are
\begin{align}
 A(\xi,e)={}&(3481875e^2-1160625)\xi^3\notag\\
 &+(-9749175e^2-7427800e+2785425)\xi^2\notag\\
 &+(9052725e^2+12998656e+1624833)\xi\notag\\
 &-2785425e^2-5570856e+434655,\label{eq:elimination-A}\\
 B(\xi,e)={}&(-1160625e^3+3481875e)\xi^3\notag\\
 &+(3481875e^3+3713900e^2-9052725e-3713900)\xi^2\notag\\
 &+(-3481875e^3-7427800e^2+3946017e+5570856)\xi\notag\\
 &+1160625e^3+3713900e^2+1624833e+163460.
 \label{eq:elimination-B}
\end{align}
Their exact resultant is
\begin{equation}\label{eq:eta-resultant}
 \operatorname{Res}_{\xi}(A,B)
 =27587592000000000000\,G(e)M(e),
\end{equation}
where
\[
 G(e)=17643e^3-70568e^2+132777e-7430
\]
and
\begin{equation}\label{eq:eta-sextic}
 \boxed{\begin{aligned}
 M(e)={}&61114327234279e^6+476161418762182e^5\\
 &+1249735047872147e^4+1103707235761676e^3\\
 &-510839674922447e^2-1100344963626938e\\
 &-597807070530939.
 \end{aligned}}
\end{equation}
The first factor is not the desired branch.  Indeed,
\[
 G(9/10)=\frac{67770967}{1000}>0,\qquad
 G'(1)=44570>0,
\]
and \(G''(e)=105858e-141136<0\) on \([9/10,1]\).  Thus \(G'\) is
decreasing there and remains at least \(G'(1)>0\), so \(G>0\) throughout
that interval.  Consequently \eqref{eq:eta-resultant} and
\eqref{eq:UV-isolation} imply \(M(\eta_*)=0\).

Let
\[
 e_-=\frac{9102483223156522812}{10^{19}},\qquad
 e_+=\frac{9102483223156522813}{10^{19}}.
\]
Exact substitution gives the useful strict sign margins
\begin{equation}\label{eq:eta-endpoint-signs}
 -\frac{278}{10^6}<M(e_-)<-\frac{277}{10^6}<0
 <\frac{356}{10^6}<M(e_+)<\frac{357}{10^6}.
\end{equation}
Moreover,
\begin{align*}
 M'(e)
 &>3(1103707235761676)(9/10)^2\\
 &\quad-2(510839674922447)-1100344963626938\\
 &=\frac{13999606735726017}{25}>0
 \qquad(9/10\leq e\leq1),
\end{align*}
where the omitted terms are positive.  Thus \(M\) has exactly one root
in this interval, and
\begin{equation}\label{eq:eta-isolation}
 e_-<\eta_*<e_+.
\end{equation}
Substitution of \eqref{eq:UV-isolation} into
\eqref{eq:arc-parameters-UV}, with rational square comparisons, also
gives
\begin{equation}\label{eq:lambda-star-isolation}
 \frac{48021524689625610102}{10^{20}}
 <\lambda_*<
 \frac{48021524689625610103}{10^{20}}.
\end{equation}

\subsection{The exact phase derivative and Descartes' rule}
\label{subsec:exact-phase-derivative}

Temporarily let \(e\in(9/10,1)\), and put
\[
 d(\lambda)=1-ie(1-\lambda),\qquad
 D(\lambda)=|d(\lambda)|^2=1+e^2(1-\lambda)^2.
\]
Because \(q^2+6q+25=0\), direct expansion gives
\begin{equation}\label{eq:arc-factorizations}
 \gamma_e(\lambda)^2+6\gamma_e(\lambda)+25
 =\frac{(1-\lambda)K(\lambda)}{d(\lambda)^2},
 \qquad
 25-\gamma_e(\lambda)=\frac{L(\lambda)}{d(\lambda)},
\end{equation}
where
\[
 \begin{split}
 K(\lambda)={}&25(1-e^2)-50ie\\
 &+\lambda\{25e^2+24e+7+i(18e+24)\},\\
 L(\lambda)={}&25d(\lambda)-q\lambda.
 \end{split}
\]
Set \(H=|K|^2\) and \(J=|L|^2\).  Their exact expansions are
\begin{align}
 H(\lambda)={}&625e^4+1250e^2+625\notag\\
 &+(-1250e^4-1200e^3-900e^2-1200e+350)\lambda\notag\\
 &+(625e^4+1200e^3+1250e^2+1200e+625)\lambda^2,
 \label{eq:H-expanded}\\
 J(\lambda)={}&625e^2+625+(-1250e^2+200e+150)\lambda\notag\\
 &+(625e^2-200e+25)\lambda^2.
 \label{eq:J-expanded}
\end{align}
The factors cleared below are strictly positive.  In fact \(D\geq1\),
while
\[
 \Im K=-50e+\lambda(18e+24)\leq24-32e<-\frac{24}{5},
 \qquad
 \Im L=-25e(1-\lambda)-4\lambda\leq-4.
\]
Thus \(H>(24/5)^2\) and \(J\geq16\) on \(0\leq\lambda\leq1\).

Using \eqref{eq:arc-factorizations}, one obtains
\[
 \begin{split}
 2F(\gamma_e(\lambda))={}&\text{constant}
 +2\alpha\log\lambda+2\beta\log(1-\lambda)
 +\beta\log H-\log J\\
 &+(1-\alpha-2\beta)\log D.
 \end{split}
\]
Differentiating and clearing the now-proved positive denominator gives
\begin{equation}\label{eq:exact-phase-derivative}
 \lambda(1-\lambda)DHJ
 \frac d{d\lambda}\{2F(\gamma_e(\lambda))\}
 =\frac{125}{557}(1+e^2)Q_e(\lambda),
\end{equation}
where \(Q_e(\lambda)=\sum_{k=0}^6q_k(e)\lambda^k\), with
\begin{align*}
 q_0={}&1160625(e^2+1)^3,\\
 q_1={}&-25(e^2+1)^2(278550e^2+185696e+92847),\\
 q_2={}&(e^2+1)(17409375e^4+23212000e^3+16031089e^2\\
 &\hspace{35mm}+15620696e+449386),\\
 q_3={}&-2(11606250e^6+23212000e^5+20456303e^4\\
 &\hspace{22mm}+29566468e^3+8066388e^2+4334052e+2900623),\\
 q_4={}&17409375e^6+46424000e^5+32353659e^4+37250632e^3\\
 &\hspace{22mm}+11447425e^2-2874872e-1000987,\\
 q_5={}&-(6963750e^6+23212000e^5+15379231e^4+6009240e^3\\
 &\hspace{25mm}+3279172e^2-564680e+92875),\\
 q_6={}&e^2(1160625e^4+4642400e^3+3264514e^2\\
 &\hspace{35mm}-1656352e-391727).
\end{align*}
This numerator genuinely has degree six.  Indeed, the bracket in \(q_6\)
is larger than
\[
 3264514(9/10)^2-1656352-391727
 =\frac{59617734}{100}>0.
\]

Now map \(0<\lambda<1\) to \(u>0\) by
\(\lambda=u/(1+u)\), and write
\begin{equation}\label{eq:Descartes-transform}
 (1+u)^6Q_e\left(\frac{u}{1+u}\right)
 =\sum_{k=0}^6C_k(e)u^k.
\end{equation}
Exact collection of coefficients gives
\begin{align}
 C_0={}&1160625(e^2+1)^3,\notag\\
 C_1={}&-175(26528e-26529)(e^2+1)^2,\notag\\
 C_2={}&2(e^2+1)(2212607e^2-3795652e+3126443),\notag\\
 C_3={}&8(418731e^3+420703e^2+923835e-500369),\notag\\
 C_4={}&-128(14851e^2-143914e+168053),\notag\\
 C_5={}&512(15784e-43639),\notag\\
 C_6={}&-7606272.\label{eq:Descartes-coefficients}
\end{align}
The following exact margins prove their signs uniformly on
\(9/10<e<1\):
\begin{gather*}
 26528e-26529<-1,\\
 2212607e^2-3795652e+3126443
 >\frac{150256787}{100}>0,\\
 923835e-500369>\frac{662165}{2}>0,\\
 14851e^2-143914e+168053>38990>0,\\
 15784e-43639<-27855<0.
\end{gather*}
For the second line, the quadratic is increasing because its derivative
at \(9/10\) is \(935203/5>0\); for the fourth, the quadratic is
decreasing and its value at \(1\) is \(38990\).  Therefore the signs of
the coefficients in ascending order are exactly
\((+,+,+,+,-,-,-)\), as used in
Proposition~\ref{prop:unique-phase-maximum}.

\subsection{Resultants and the saddle-value enclosure}
\label{subsec:saddle-value-enclosure}

Let
\[
 m_y=|y_c|^2,\qquad
 m_A=|y_c^2+6y_c+25|^2,\qquad
 m_J=|25-y_c|^2.
\]
The product of the three roots of \(p\), together with
\eqref{eq:Vieta-complex-root}, gives
\begin{equation}\label{eq:my-resultant-value}
 m_y=\frac{232125}{743\rho}.
\end{equation}
For \(f(y)=y^2+6y+25\), Euclidean division gives
\[
 p(y)\equiv505104y+5199600\pmod{f(y)}.
\]
If \(u,v\) are the two roots of \(f\), then \(u+v=-6\) and \(uv=25\).
It follows directly that
\begin{align*}
 \operatorname{Res}_y(p,f)
 &=5199600^2-6(505104)(5199600)+25(505104)^2\\
 &=17656058880000,\\
 \prod_{p(r)=0}f(r)&=\frac{706242355200}{552049}.
\end{align*}
Dividing out the real-root factor gives
\begin{equation}\label{eq:mA-resultant-value}
 m_A=\frac{706242355200}
 {552049(\rho^2+6\rho+25)}.
\end{equation}
Finally \(p(25)=-111400000\), and therefore
\begin{equation}\label{eq:mJ-resultant-value}
 m_J=\frac{p(25)}{3715(25-\rho)}
 =\frac{22280000}{743(\rho-25)}.
\end{equation}
Substitution of the rational endpoints \(\rho_L,\rho_U\), with the
monotone orientation reversed where necessary, yields
\begin{equation}\label{eq:modulus-enclosures}
 \begin{aligned}
 \frac{4710661776618520503155596}{10^{24}}
 &<m_y<
 \frac{4710661776618520503155597}{10^{24}},\\
 \frac{265339959117980760581398972}{10^{24}}
 &<m_A<
 \frac{265339959117980760581398973}{10^{24}},\\
 \frac{725697065335997235474816912}{10^{24}}
 &<m_J<
 \frac{725697065335997235474816913}{10^{24}}.
 \end{aligned}
\end{equation}
The phase value is
\begin{equation}\label{eq:s-from-moduli}
 s=\frac{1857}{11140}\log m_y
   +\frac{1857}{5570}\log m_A-\frac12\log m_J.
\end{equation}
Using \(N=120\) in \eqref{eq:atanh-log-bound}, evaluating at the exact
outer endpoints in \eqref{eq:modulus-enclosures}, and reversing the
endpoints for the negative last term gives an interval of width less
than \(3.74\cdot10^{-26}\), contained strictly in
\begin{equation}\label{eq:s-rational-enclosure}
 -\frac{1174543941310025108442583231}{10^{27}}
 <s<
 -\frac{1174543941310025108442583192}{10^{27}}.
\end{equation}
This is \eqref{eq:s-enclosure} and fixes the safe rounding direction for
the later irrationality-measure estimate.

\subsection{A rigorous sign calculation at the old point}
\label{subsec:old-point-sign}

We record a short exact enclosure needed only for the one-sided
comparison with the Zeilberger--Zudilin point.  At
\(p_{\rm ZZ}=(1/3,2/3)\), the stationary polynomial, after multiplication
by \(3\), is
\[
 p_0(y)=2y^3-125y^2-500y-625.
\]
It has one positive root \(\rho_0\), by Descartes' rule, and exact
substitution of the terminating decimals
\[
 r_- =66.33950152462140818,\qquad
 r_+ =66.33950152462140819
\]
gives
\[
 p_0(r_-)<0<p_0(r_+).
\]
More explicitly, the two exact values are
\[
 -\frac{2206947636853826486502861778769572071}
 {62500000000000000000000000000000000000000000000000}
\]
and
\[
 \frac{28947925875450012108168704650214973259}
 {500000000000000000000000000000000000000000000000000},
\]
respectively.  Thus \(r_-<\rho_0<r_+\).

Vieta's formulas and one elementary resultant calculation give
\begin{equation}\label{eq:old-point-moduli}
 m_y=\frac{625}{2\rho_0},\qquad
 m_A=\frac{1280000}{\rho_0^2+6\rho_0+25},\qquad
 m_J=\frac{30000}{\rho_0-25}.
\end{equation}
Indeed,
\(\operatorname{Res}(p_0,y^2+6y+25)=5120000\), while
\(p_0(25)=-60000\).  It follows that
\begin{align}
 r_{\rm ZZ}
 &=\frac13\log\rho_0+
   \frac23\log(\rho_0^2+6\rho_0+25)-\log(\rho_0-25),
 \label{eq:old-r-formula}\\
 s_{\rm ZZ}
 &=\frac16\log m_y+\frac13\log m_A-\frac12\log m_J.
 \label{eq:old-s-formula}
\end{align}
The old saving and normalization rate are
\begin{align}
 \omega_{\rm ZZ}
 &=\frac{\pi}{2\sqrt3}-\log\frac{3\sqrt3}{4},
 \label{eq:old-omega-formula}\\
 \mathcal C_{\rm ZZ}
 &=\frac{4-\frac52\log2-\omega_{\rm ZZ}}3.
 \label{eq:old-C-formula}
\end{align}

Here is a wholly rational way to bound the only new constant \(\pi\)
in \eqref{eq:old-omega-formula}.  Machin's identity is
\[
 \pi=16\arctan\frac15-4\arctan\frac1{239}.
\]
For \(0<x<1\), the alternating series
\[
 A_N(x)=\sum_{j=0}^{N-1}\frac{(-1)^jx^{2j+1}}{2j+1}
\]
has error of sign \((-1)^N\) and magnitude less than
\(x^{2N+1}/(2N+1)\).  Taking \(N=40\), and using the exact square
comparisons
\[
 1.7320508075688772935<\sqrt3<
 1.7320508075688772936,
\]
bounds every nonrational quantity in
\eqref{eq:old-r-formula}--\eqref{eq:old-C-formula}.  Applying the finite
logarithm sum \eqref{eq:atanh-log-bound} with \(N=120\), always choosing
the endpoint dictated by monotonicity, gives by rational arithmetic
\begin{equation}\label{eq:old-coarse-rate-box}
 \begin{gathered}
 3.3306<r_{\rm ZZ}<3.3308,\qquad
 -1.1750<s_{\rm ZZ}<-1.1748,\\
 0.5405<\mathcal C_{\rm ZZ}<0.5407.
 \end{gathered}
\end{equation}
All series and their remainder signs have been printed above, so
\eqref{eq:old-coarse-rate-box} is a finite rational calculation.
In particular,
\[
 0.6341<\tau_{\rm ZZ}:=-s_{\rm ZZ}-\mathcal C_{\rm ZZ}<0.6345,
\qquad
 4.5054<r_{\rm ZZ}-s_{\rm ZZ}<4.5058.
\]
Finally,
\[
 4.5054-11(0.6345)^2=0.07690725>0,
\qquad
 12(0.6341)^2-4.5058=0.31919372>0.
\]
These exact decimal identities prove
\[
 11<
 \frac{r_{\rm ZZ}-s_{\rm ZZ}}{\tau_{\rm ZZ}^2}
 <12,
\]
which is the rigorous sign assertion used in
\eqref{eq:K-ZZ-rigorous-sign}.

\section*{Acknowledgments}
The author is grateful to Professor Bin Dong and Haocheng Ju of Peking University
for their guidance.

\bibliographystyle{amsplain}
\bibliography{references}

\end{document}